\documentclass[11pt]{amsart}
\usepackage{amsfonts} 
\usepackage{amssymb}
\usepackage{amsthm}
\usepackage{amsmath} 
\usepackage{mathrsfs}
\usepackage{mathtools}
\usepackage{dsfont}
\usepackage{esint}
\usepackage{ulem}
\usepackage{marginnote}
\usepackage{array}
\usepackage{algorithm}
\usepackage{algpseudocode}
\usepackage{soul}
\usepackage{comment}

\usepackage{booktabs} 

\usepackage{MnSymbol}

\usepackage{longtable}

\usepackage[all]{xy}

\usepackage[english]{babel} 
\usepackage[left=3.5cm,right=3.5cm,top=3.5cm,bottom=3cm,headsep=0.7cm]{geometry}
\usepackage{xargs}

\usepackage{csquotes}
\usepackage{stmaryrd}
\usepackage[colorinlistoftodos,prependcaption,textsize=tiny]{todonotes}
\newcommandx{\unsure}[2][1=]{\todo[linecolor=red,backgroundcolor=red!25,bordercolor=red,#1]{#2}}
\newcommandx{\change}[2][1=]{\todo[linecolor=blue,backgroundcolor=blue!25,bordercolor=blue,#1]{#2}}
\newcommandx{\info}[2][1=]{\todo[linecolor=OliveGreen,backgroundcolor=OliveGreen!25,bordercolor=OliveGreen,#1]{#2}}
\newcommandx{\improvement}[2][1=]{\todo[linecolor=Plum,backgroundcolor=Plum!25,bordercolor=Plum,#1]{#2}}
\newcommandx{\thiswillnotshow}[2][1=]{\todo[disable,#1]{#2}}

\usepackage{pgf,tikz}
\usetikzlibrary{positioning, calc, arrows.meta}
\usetikzlibrary{arrows}
\usepackage{subcaption}
\usepackage{graphicx}
\usepackage[parfill]{parskip}

\usepackage{enumerate}
\usepackage{enumitem}

\usepackage{hyperref}
\usepackage[nameinlink]{cleveref}
\usepackage{mathtools}
\hypersetup{
    colorlinks,
    linkcolor={red!80!black},
    citecolor={blue!80!black}
}

\crefname{assumption}{assumption}{assumptions}  
\Crefname{assumption}{Assumption}{Assumptions}  
\crefname{counterexample}{counterexample}{counterexamples}
\Crefname{counterexample}{Counterexample}{Counterexamples}

\crefname{problem}{problem}{problems}
\Crefname{problem}{Problem}{Problems}
\crefname{table}{table}{tables}
\Crefname{table}{Table}{Tables}

\theoremstyle{plain}
\begingroup
\theoremstyle{plain}
\newtheorem{theorem}{Theorem}[section]
\newtheorem{corollary}[theorem]{Corollary}
\newtheorem{proposition}[theorem]{Proposition}
\newtheorem{lemma}[theorem]{Lemma}
\theoremstyle{definition}
\newtheorem{definition}[theorem]{Definition}

\newtheorem{problem}[theorem]{Problem}
\theoremstyle{remark}
\newtheorem{remark}[theorem]{Remark}

\endgroup

\theoremstyle{definition}
\theoremstyle{remark}

\numberwithin{equation}{section}

\theoremstyle{definition}
\theoremstyle{remark}

\numberwithin{equation}{section}

\newcommand{\X}{\mathcal{X}}
\newcommand{\K}{\mathcal{K}}
\newcommand{\Y}{\mathcal{Y}}

\DeclareMathOperator{\Lip}{Lip}

\makeindex 

\title[\bf SCI Of Koopman Approximate Point Spectra In \texorpdfstring{$L^{\MakeLowercase{p}}$}{Lp}, \texorpdfstring{$1<\MakeLowercase{p}<\infty$}{1<p<infty}]{\bf Residual SCI Upper Bounds And Lower Witnesses For Koopman Approximate Point Spectra In \texorpdfstring{$L^{\MakeLowercase{p}}$}{Lp} For \texorpdfstring{$1<\MakeLowercase{p}<\infty$}{1<p<infty}: Extended Version}

\begin{document}

\author[C.~Sorg]{Christopher Sorg$^1$}
	\address[C.~Sorg]{
		\textup{Chair for Theoretical computer science, mathematics, and operations research}
		\newline \indent
		\textup{Department of Computer Science} \newline \indent
		\textup{University of the Bundeswehr Munich}
		\newline \indent
		\textup{85577 Neubiberg, Germany}}

\email{{\href{mailto:chr.sorg@unibw.de}{\textcolor{blue}{\texttt{chr.sorg@unibw.de}}}}
}

\footnotetext[1]{Inf1, University of the Bundeswehr Munich, Werner-Heisenberg-Weg 39, 85577 Neubiberg, Germany}

\begin{abstract}
We study residual computation of approximate point spectral sets of bounded Koopman operators \texorpdfstring{$\mathcal K_F$}{KF} on \texorpdfstring{\(L^p(\mathcal X,\omega)\)}{Lp(X,w)}, \texorpdfstring{\(1<p<\infty\)}{1<p<oo}, where \texorpdfstring{\(\mathcal X\)}{X} is a compact metric space and \texorpdfstring{\(\omega\)}{w} is a finite Borel measure. The input is the underlying map \texorpdfstring{\(F : \mathcal X \to \mathcal X\)}{F:X->X}, accessed through point evaluations, and the output metric is the Hausdorff metric on non-empty compact subsets of \texorpdfstring{\(\mathbb C\)}{C}. For a bounded operator \texorpdfstring{\(T\)}{T}, we distinguish the regularized approximate point \texorpdfstring{\(\varepsilon\)}{epsilon}-pseudospectrum \texorpdfstring{$R_{\mathrm{ap},\varepsilon}(T)$}{Rap,epsilon(T)} from the closed approximate point \texorpdfstring{\(\varepsilon\)}{epsilon}-pseudospectrum \texorpdfstring{$C_{\mathrm{ap},\varepsilon}(T)$}{Cap,epsilon(T)}. The latter is the direct closed lower-norm analogue of the approximate point \texorpdfstring{\(\varepsilon\)}{epsilon}-pseudospectrum used in the \texorpdfstring{\(L^2\)}{L2} Koopman SCI theory. Using continuous finite-dimensional dictionaries and tagged quadrature residuals, we prove SCI upper bounds for \texorpdfstring{\(R_{\mathrm{ap},\varepsilon}(T)\)}{Rap,epsilon(T)}, \texorpdfstring{\(C_{\mathrm{ap},\varepsilon}(T)\)}{Cap,epsilon(T)}, and \texorpdfstring{\(\sigma_{\mathrm{ap}}\)}{specap} on four natural classes of maps: continuous nonsingular maps, maps with a prescribed modulus of continuity, measure-preserving maps, and maps satisfying both measure preservation and a prescribed modulus. In the measure-preserving case the two fixed-\texorpdfstring{\(\varepsilon\)}{epsilon} targets coincide, because the Koopman operator is an \(L^p\)-isometry. We also prove fixed-\texorpdfstring{\(\varepsilon\)}{epsilon} sharpness on a known-modulus measure-preserving witness class and a boundary lower obstruction for the closed nonsingular problem. Finally, we identify the remaining no-modulus sharpness problem for the exact approximate point spectrum and explain why the natural locking strategy cannot prove it.
\end{abstract}

\maketitle

\begin{center}\small
\textbf{Keywords:} Koopman operators \(\cdot\) approximate point spectrum \(\cdot\) pseudospectra \(\cdot\) \texorpdfstring{\(L^p\)}{Lp}-spaces \(\cdot\) residual algorithms \(\cdot\) Solvability Complexity Index
\end{center}
\tableofcontents

\section{Introduction}\label{sec:intro}

Koopman's operator-theoretic formulation represents nonlinear dynamics by a linear operator acting on observables rather than directly on states. Its Hilbert-space origins go back to Koopman's paper
\cite{koopman1931hamiltonian}, the Koopman-von Neumann work on dynamical systems with continuous spectra \cite{koopman1932dynamical}, and von Neumann's operator-method formulation of classical mechanics \cite{vonneumann1932operatorenmethode}. The modern data-driven revival of this viewpoint, strongly influenced by Mezi\'c's work \cite{mezic2005spectral}, has led to algorithms such as Dynamic Mode Decomposition (DMD), Extended Dynamic Mode Decomposition (EDMD), and many related variants; see, for example, \cite{rowley2009spectral,schmid2010dynamic,budivsic2012applied,brunton2022modern,
colbrook2024multiverse,colbrook2025introguide}.

There has also been substantial progress on rigorous convergence and residual control for data-driven Koopman approximations. Arbabi and Mezi\'c proved convergence of a class of DMD/Hankel-DMD algorithms for computing Koopman eigenvalues and eigenfunctions under ergodicity hypotheses \cite{arbabi2017ergodic}. Korda, Putinar, and Mezi\'c developed moment-based data-driven spectral-measure methods for measure-preserving systems \cite{korda2020data}. ResDMD, introduced by Colbrook and Townsend, provides a residual-based method with convergence guarantees for computing spectra and pseudospectra of general Koopman operators from snapshot data without spectral pollution \cite{colbrook2024rigorous}. Related residual and verification ideas appear in robust and verified Koopmanism \cite{colbrook2023residual}. For spectral measures and generalized eigenfunction decompositions, Rigged DMD provides data-driven generalized eigenfunction expansions and convergence results for spectral measures \cite{colbrook2025rigged}. Variance-based residual ideas for stochastic Koopman operators were developed in \cite{colbrook2024beyond}, and recent variance representations and convergence-rate estimates for EDMD under i.i.d.\ and ergodic sampling appear in \cite{philipp2024variance}. Auxiliary functions as Koopman observables provide another data-driven route, connected with polynomial optimization, for dynamical-systems \cite{bramburger2024auxiliary}. 

The recent surveys \cite{colbrook2024multiverse,colbrook2025introguide} give broader accounts of DMD-type algorithms, residual methods, and rigorously convergent Koopman learning.

Despite these developments, a basic foundational question remains open outside the Hilbert-space setting:
\begin{align*} \textbf{Q:} \text{ What is } &\text{the intrinsic algorithmic complexity of approximating Koopman}\\ &\text{spectra in a way that is robust to noise and admits verification? } \end{align*}
Here algorithmic complexity is understood in the sense of the \textit{Solvability Complexity Index} (SCI), which measures the minimal number of nested limits required by any admissible algorithm to solve e.g. an infinite-dimensional computational problem \cite{hansen2011solvability,ben2015new,colbrook2022foundations}. The SCI framework has revealed sharp barriers for spectral computation: some spectral problems admit no single-limit procedure with global error control, some become solvable only by finite-height towers of algorithms, and others may have infinite SCI; see, for example, the operator-theoretic classifications in \cite{hansen2011solvability,ben2015new,colbrook2022foundations} and the Koopman \(L^2\) results in \cite{colbrook2024limits}.

This paper studies bounded Koopman operators
\[
        \K_F f=f\circ F
\]
on \(L^p(\X,\omega)\), \(1<p<\infty\), where \((\X,d_\X)\) is a compact metric space, $F:\X \to \X$ and \(\omega\) is a finite Borel measure. The SCI computation model is a point-evaluation oracle for the underlying map \(F\), and the output metric is the Hausdorff metric on non-empty compact subsets of \(\mathbb C\). We work with four natural input classes: continuous nonsingular maps, maps with a prescribed modulus of continuity, measure-preserving maps, and maps satisfying both measure preservation and a prescribed modulus.

\section{Preliminaries And Definitions}\label{sec:preliminaries}

This section recalls only the SCI terminology needed later. The Koopman
operator and the approximate point spectral targets are introduced in
\cref{sec:KoopLp}.  We work throughout with set-valued
outputs in the Hausdorff metric space \(\mathcal M_H\) of non-empty compact subsets of
\(\mathbb C\).  This is the metric setting in which convergence excludes both
spectral pollution and spectral invisibility.

The definitions below follow the general SCI framework used in
\cite{colbrook2024limits}.  We state them only in the form needed for the present paper,
namely exact point-evaluation oracles and their \(\Delta_1\)-information
variant.

\subsection{General Algorithms And Towers Of Algorithms}\label{subsec:towers}

Formally, given a set $\Omega$ and $\Lambda$, a \textit{general algorithm} is a map
$\Gamma\colon\Omega\rightarrow\mathcal{M}$ that queries only finitely many
evaluations from~$\Lambda$ on each input, subject to the consistency ''axiom'' of \cite[Definition A.1]{colbrook2024limits}, because otherwise it is impossible to recover a problem function $\Xi$ from an evaluation set $\Lambda$, which will be precised in the following.

\begin{definition}[Computational problem {\cite[Definition A.1]{colbrook2024limits}}]
\label{def:comp_prob}
The basic objects of a computational problem are:
\begin{itemize}
	\item A \textit{primary set}, $\Omega$, that describes the input class;
	\item A \textit{metric space} $(\mathcal{M},d)$;
	\item A \textit{problem function} $\Xi:\Omega\rightarrow\mathcal{M}$;
	\item An \textit{evaluation set}, $\Lambda$, of complex-valued functions on $\Omega$.
\end{itemize}
The problem function $\Xi$ is the object we want to compute, with the notion of convergence captured by the metric space $(\mathcal{M},d)$. The evaluation set $\Lambda$ describes the information that algorithms can read. We require that $\Lambda$ separates elements of $\Omega$ to the degree of separation achieved by $\Xi$:
\begin{equation}
\label{eq:consistencyAxiom}
\text{if }A,B\in\Omega\text{ with }\Xi(A)\neq \Xi(B),\text{ then } \exists f\in\Lambda\text{ with }f(A)\neq f(B).
\end{equation}
In other words, any $\Xi(A)\in\mathcal{M}$ is uniquely determined by the set of evaluations $\{f(A):f\in\Lambda\}$, an axiom, which we will call consistency axiom in the following. It is necessary, because otherwise it is impossible to recover $\Xi$ from $\Lambda$, as stated before. We refer to the collection $\{\Xi,\Omega,\mathcal{M},\Lambda\}$ as a \textit{computational problem}.
\end{definition}

By an algorithm we mean then the following.
\begin{definition}[General algorithm  {\cite[Definition A.3]{colbrook2024limits}}]
\label{def:Gen_alg}
Given a computational problem $\{\Xi,\Omega,\mathcal{M},\Lambda\}$, a {general algorithm} is a map $\Gamma:\Omega\to \mathcal{M}$ such that $\forall A\in\Omega$,
\begin{itemize}
\item[1.] There exists a non-empty finite subset of evaluations $\Lambda_\Gamma(A) \subset\Lambda$;
\item[2.] The action of $\Gamma$ on $A$ only depends on $\{f(A)\}_{f \in \Lambda_\Gamma(A)}$;
\item[3.] If $B\in\Omega$ with $f(A)=f(B)$ for every $f\in\Lambda_\Gamma(A)$, then $\Lambda_\Gamma(A)=\Lambda_\Gamma(B)$ and $\Gamma(A)=\Gamma(B)$.
\end{itemize}
\end{definition}

To capture multi‑limit procedures we recall:

\begin{definition}[Tower of algorithms {\cite[Definition A.4]{colbrook2024limits}}]{\label{def:tower_funct}}
Let $k\in\mathbb{N}$.
A tower of algorithms of height $k$ for a computational problem $\{\Xi,\Omega,\mathcal{M},\Lambda\}$ is a collection of functions
$$
\Gamma_{n_k},\,
\Gamma_{n_k, n_{k-1}},\, \ldots,\,\Gamma_{n_k, \ldots, n_1}:\Omega \rightarrow \mathcal{M}, \quad n_k,\ldots,n_1 \in \mathbb{N},
$$
where $\{\Gamma_{n_k, \ldots, n_1}\}$ are general algorithms (\cref{def:Gen_alg}) and for every $A \in \Omega$, the following convergence holds in $(\mathcal{M},d)$:
\begin{equation*}
\begin{split}
&\Xi(A) = \lim_{n_k \rightarrow \infty} \Gamma_{n_k}(A), \quad\Gamma_{n_k}(A) =
  \lim_{n_{k-1} \rightarrow \infty} \Gamma_{n_k, n_{k-1}}(A),\quad\ldots,\\
&\Gamma_{n_k, \ldots, n_2}(A) =
  \lim_{n_1 \rightarrow \infty} \Gamma_{n_k, \ldots, n_1}(A).
\end{split}
\end{equation*}
\end{definition}

When establishing upper bounds for the SCI, we can determine the type of tower by setting constraints on the functions $\{\Gamma_{n_k, \ldots, n_1}\}$ at the base level. Axiomatically this gives us two possible types, as also analyzed in e.g. \cite{colbrook2024limits}:
\begin{itemize}
\item A \textit{general tower}, denoted by $\alpha=G$, refers to \cref{def:tower_funct} with no further restrictions.
\item An \textit{arithmetic tower}, denoted by $\alpha=A$, refers to \cref{def:tower_funct} where each $\Gamma_{n_k, \ldots, n_1}(B)$ can be computed using $\Lambda$ and finitely many arithmetic operations and comparisons. More precisely, if $\Lambda$ is countable, each output $\Gamma_{n_k, \ldots, n_1}(B)$ is a finite string of numbers (or encoding) that can be identified with an element in $\mathcal{M}$, and the following function is recursive:
$$
(n_k, \ldots, n_1,\{f(B)\}_{f \in \Lambda})\mapsto \Gamma_{n_k, \ldots, n_1}(B)
$$

\end{itemize}

\begin{definition}[Solvability Complexity Index {\cite[Definition A.6]{colbrook2024limits}}]
A computational problem $\{\Xi,\Omega,\mathcal{M},\Lambda\}$ has Solvability Complexity Index $k\in\mathbb{N}$ with respect to type $\alpha$, written $\mathrm{SCI}(\Xi,\Omega,\mathcal{M},\Lambda)_{\alpha} = k$, if $k$ is the smallest integer for which there exists a tower of algorithms of type $\alpha$ and height $k$ for $\Xi$. If no such tower exists, then $\mathrm{SCI}(\Xi,\Omega,\mathcal{M},\Lambda)_{\alpha} = \infty.$ If there exists an algorithm $\Gamma$ of type $\alpha$ with $\Xi = \Gamma$, then $\mathrm{SCI}(\Xi,\Omega,\mathcal{M},\Lambda)_{\alpha} = 0$.
\end{definition}

We use the same \(\Delta_1\)-information convention as in \cite{colbrook2024limits}: exact oracle values may be replaced by approximants with certified error \(2^{-n}\).

The SCI induces a hierarchy in natural manner as follows.

\begin{definition}[The SCI Hierarchy {\cite[Definition 2.6]{ben2022universal}}]
The SCI Hierarchy is a hierarchy $\{\Delta_k\}_{k \in \mathbb{N}_0}$ of classes of computational problems $(\Omega,\Lambda,\Xi,\mathcal{M})$, where each $\Delta_k$ is defined as the collection of all computational problems satisfying
$$
\begin{aligned}
    &(\Omega,\Lambda,\Xi,\mathcal{M}) \in \Delta_0 \quad \Leftrightarrow \quad \text{SCI}(\Omega,\Lambda,\Xi,\mathcal{M}) = 0,\\
    &(\Omega,\Lambda,\Xi,\mathcal{M}) \in \Delta_{k+1} \quad \Leftrightarrow \quad \text{SCI}(\Omega,\Lambda,\Xi,\mathcal{M}) \leq k,\\
\end{aligned}
$$
where $k \in \mathbb{N}$. Here, the class $\Delta_1$ is a special case, defined as the class of all computational problems in $\Delta_2$ with a convergence rate, i.e.
\[
(\Omega,\Lambda,\Xi,\mathcal{M}) \in \Delta_1 \quad \Leftrightarrow \quad (\exists \{\Gamma_n\}_{n \in \mathbb{N}}) (\exists \varepsilon_n \downarrow 0), \text{ s.t. } \forall B \in \Omega \, : \, d(\Gamma_n(B),\Xi(B)) \leq \varepsilon_n.
\]
\end{definition}
\subsubsection{Refinements That Capture Error Control}
Sufficient structure in $(\mathcal{M},d)$ enables two types of verification or error control: convergence from above and below. An important case used in this work is the Hausdorff metric space, $(\mathcal{M}_H,d_H)$, that is suitable for computing spectra of bounded operators. It is the collection of non-empty compact subsets of $\mathbb{C}$ equipped with the Hausdorff metric:
\begin{equation}\label{eq:Hausdorff}
d_{H}(X,Y) = \max\left\{\sup_{x \in X} \inf_{y \in Y} |x-y|, \sup_{y \in Y} \inf_{x \in X} |x-y| \right\},\quad X,Y\in\mathcal{M}_H.
\end{equation}
We are interested in the Hausdorff metric since convergence to the spectrum in this metric means that our algorithms converge without spectral pollution or spectral invisibility, see \cite{Pokrzywa79} for a proof.

\begin{definition}[$\Sigma$ and $\Pi$ classes for Hausdorff metric {\cite[Definition A.11]{colbrook2024limits}}]
\label{def:sig_pi_spec}
Consider a collection $\mathcal{C}$ of computational problems and let $\mathcal{T}_\alpha$ be the collection of all towers of algorithms of type $\alpha$. Suppose that $(\mathcal{M},d)$ is the Hausdorff metric. We set $\Sigma^{\alpha}_0 = \Pi^{\alpha}_0 = \Delta^{\alpha}_0$ and for $m \in \mathbb{N}$, we define
\begin{equation*}
\begin{split}
\Sigma^{\alpha}_{m}&= \Big\{\{\Xi,\Omega,\mathcal{M},\Lambda\} \in \Delta_{m+1}^\alpha  :  \exists \{\Gamma_{n_{m},\ldots,n_1}\} \in \mathcal{T}_\alpha,  \{X_{n_{m}}(A)\}\subset\mathcal{M} \text{ s.t. }\forall A \in\Omega\\
&\quad\quad\quad\quad\quad\quad\Gamma_{n_{m}}(A)  \subset X_{n_{m}}(A),\lim_{n_{m}\rightarrow\infty}\Gamma_{n_{m}}(A)=\Xi(A), d(X_{n_{m}}(A),\Xi(A))\leq 2^{-n_{m}}\Big\},\\
\Pi^{\alpha}_{m} &= \Big\{\{\Xi,\Omega,\mathcal{M},\Lambda\} \in \Delta_{m+1}^\alpha  :  \exists \{\Gamma_{n_{m},\ldots,n_1}\} \in \mathcal{T}_\alpha,  \{X_{n_{m}}(A)\}\subset\mathcal{M} \text{ s.t. }\forall A \in \Omega\\
&\quad\quad\quad\quad\quad\quad\Xi(A)  \subset X_{n_{m}}(A),\lim_{n_{m}\rightarrow\infty}\Gamma_{n_{m}}(A)=\Xi(A),
 d(X_{n_{m}}(A),\Gamma_{n_{m}}(A))\leq 2^{-n_{m}} \Big\}.
\end{split}
\end{equation*}
\end{definition}

The distinction between \(\Sigma_m^G\) and \(\Pi_m^G\) is important below.
The regularized target \(R_{\mathrm{ap},\varepsilon}\) is computed by inner
sublevel tests and therefore naturally belongs to \(\Sigma\)-classes.  The
closed target \(C_{\mathrm{ap},\varepsilon}\) is a closed lower-norm sublevel set;
outside the measure-preserving case it is obtained from the regularized targets
by an additional outer intersection, and therefore naturally belongs to a
\(\Pi\)-class.

\section{Koopman Framework On \texorpdfstring{$L^{p}(\mathcal X,\omega)$}{Lp(X,w)}}\label{sec:KoopLp}

Throughout this section, $(\mathcal X,d_{\mathcal X})$ is a compact metric space, \(\omega\) is a finite Borel measure on \(\mathcal X\), and
\(1<p<\infty\).  We assume \(\omega(X)>0\), since otherwise the problem is trivial, and write as usual
\[
        \|f\|_p:=\|f\|_{L^p(\mathcal X,\omega)}.
\]

\subsection{Boundedness Of The Koopman Operator}\label{sec:KoopmanLpBound}

Given a measurable map $F:\mathcal X\to\mathcal X$, define
\[
  \mathcal K_F : L^p(\mathcal X,\omega) \to L^p(\mathcal X,\omega), \quad
  (\mathcal K_F f)(x) := f\bigl(F(x)\bigr),\;\;f\in L^p(\mathcal X,\omega).
\]
It is a bounded operator on $L^p(\mathcal X,\omega)$ precisely under the usual Radon-Nikodym boundedness condition.

\begin{proposition}[Boundedness criteria]\label{prop:boundednessLp}
Let $F : \mathcal X \to \mathcal X$ be measurable and nonsingular with respect to $\omega$ (i.e.\ $F_{\#}\omega\ll\omega$) and denote 
\(
  \rho_F(x) := \dfrac{dF_{\#}\omega}{d\omega}(x).
\)
Then
\[
  \|\mathcal K_F\|_{L^p \to L^p} = \bigl\|\rho_{F}\bigr\|_{L^{\infty}}^{\,1/p}.
\]
Consequently
\begin{enumerate}
  \item[i)] If $F$ is \textit{measure‑preserving} ($\rho_{F}=1$ a.e.), then $\mathcal K_{F}$ is an \textit{isometry} on $L^{p}(\mathcal{X},\omega)$.
  \item[ii)] If $\rho_{F}\in L^{\infty}(\mathcal{X},\omega)$, then $\mathcal K_{F}$ is bounded with norm $\le\|\rho_{F}\|_{\infty}^{1/p}$.
  \item[iii)] If $\rho_{F}\notin L^{\infty}(\mathcal{X},\omega)$, $\mathcal K_{F}$ is \textit{unbounded}.
\end{enumerate}
\end{proposition}

\begin{proof}
This result is very classical, therefore we repeat just the main idea of the proof. For \(f\in L^p(\mathcal X,\omega)\),
\[
\|\mathcal K_F f \|_p^p = \int_X |f\circ F|^p\,d\omega = \int_X |f|^p\,dF_\#\omega = \int_X |f|^p\rho_F\,d\omega .
\]
Thus
\[
\|\mathcal K_F f \|_p^p\le \|\rho_F\|_\infty\|f\|_p^p .
\]
Conversely, for every \(\eta>0\) choose a measurable set \(E\) of positive measure such that
\[
\rho_F\ge \|\rho_F\|_\infty-\eta \quad\text{on }E.
\]
Testing with $f=\omega(E)^{-1/p} \mathbf 1_E$ gives
\[
\|\mathcal K_Ff\|_p^p = \omega(E)^{-1}\int_E\rho_F\,d\omega \ge \|\rho_F\|_\infty-\eta .
\]
Letting \(\eta\downarrow0\) proves
\[
\|\mathcal K_F\|_{L^p\to L^p}=\|\rho_F\|_\infty^{1/p}.
\]

If $\rho_F\notin L^\infty$, then for every $M>0$ there exists a measurable set $E_M$ of positive measure such that
\[
   \rho_F\ge M^p \quad\text{on }E_M .
\]
Testing with $f_M=\omega(E_M)^{-1/p}\mathbf 1_{E_M}$ gives
\[
   \|f_M\|_p=1,\quad \|\K_F f_M\|_p^p = \omega(E_M)^{-1} \int_{E_M}\rho_F\,d\omega \ge M^p .
\]
Hence $\|\K_F f_M\|_p\ge M$ for every $M>0$, so $\K_F$ is unbounded.
\end{proof}

In the following by a modulus of continuity we mean a function $\alpha:[0,\infty)\to[0,\infty)$ such that
\[
        \alpha(0)=0,\, \lim_{t\downarrow 0}\alpha(t)=0,
\]
and, after replacing \(\alpha\) by its monotone envelope if necessary, we assume \(\alpha\) is nondecreasing.

\begin{definition}[Input classes]\label{def:input-classes}
Let
\[
\Omega_\X := \left\{ F: \X \to \X: F\text{ is continuous},\ F_\#\omega\ll\omega,\ \frac{dF_\#\omega}{d\omega}\in L^\infty(\X,\omega)\right\}.
\]
For a prescribed modulus of continuity \(\alpha\), define
\[
\Omega_\X^\alpha := \left\{ F\in\Omega_\X: d_\X(Fx,Fy)\le \alpha(d_\X(x,y))\ \forall x,y\in \X \right\}.
\]
The measure-preserving classes are
\[
\Omega_\X^m := \{F: \X \to \X : F\text{ is continuous and }F_\#\omega=\omega\},
\]
and
\[
\Omega_\X^{\alpha,m}:=\Omega_\X^\alpha\cap\Omega_\X^m .
\]
\end{definition}

For the corresponding computational problems below we will use the point-evaluation representation of $F$. Formally, to keep the evaluation set complex-valued as required in \cref{def:comp_prob}, fix a countable family $\mathcal A \subset C(X,\mathbb{C})$ separating points of $\X$ and containing all dictionary functions used below. We take then
\[
\Lambda_\X := \{\lambda_{x,\varphi} : F \mapsto \varphi(F(x)) : x \in \X, \, \varphi \in \mathcal A \}.
\]
Thus a finite point query at $x$ is implemented by finitely many complex queries $\lambda_{x,\varphi}$. In the residual algorithms below, only values of the form
\[
\psi_{n,j}(F(x_{m,P})) = \lambda_{x_{m,P},\psi_{n,j}}(F)
\]
are used. Further note that for \(F\in\Omega_\X\), \cref{prop:boundednessLp} implies that \(\K_F\) is a bounded operator on \(L^p(\X,\omega)\). If \(F\in\Omega_\X^m\), then \(\K_F\) is an \(L^p\)-isometry.

\subsection{Approximate Point Spectral Sets}

The fixed-\(\varepsilon\) problem requires a distinction that is invisible in some Hilbert-space settings. The strict lower-norm sublevel set is open, whereas our output space \(\mathcal M_H\) consists of non-empty compact subsets of \(\mathbb C\). We therefore separate the compact regularization of the strict sublevel set from the closed lower-norm sublevel set.

\begin{definition}[Approximate point spectral sets]\label{def:ap-targets}
Let \(T\in\mathcal B(\Y)\), where \(\Y\) is a complex Banach space. Define
\[
        \nu_T(z):=\sigma_{\inf}(T-zI) := \inf_{\|u\|_\Y=1}\|(T-zI)u\|_\Y .
\]
The approximate point spectrum is
\[
        \sigma_{\mathrm{ap}}(T) := \{z\in\mathbb C:\nu_T(z)=0\}.
\]
For \(\varepsilon>0\), define the regularized approximate point \(\varepsilon\)-pseudospectrum by
\[
        R_{\mathrm{ap},\varepsilon}(T) := \overline{\{z\in\mathbb C:\nu_T(z)<\varepsilon\}},
\]
and the closed approximate point \(\varepsilon\)-pseudospectrum by
\[
        C_{\mathrm{ap},\varepsilon}(T) := \{z\in\mathbb C:\nu_T(z)\le\varepsilon\}.
\]
\end{definition}

\begin{lemma}[Basic properties of the lower-norm targets]\label{lem:basic-lower-norm-relations}
Let \(T \in \mathcal B(\Y)\). Then the following hold.
\begin{enumerate}
\item[i)] The map \(z\mapsto\nu_T(z)\) is \(1\)-Lipschitz.
\item[ii)] The closed target admits the sequence characterization
\[
        C_{\mathrm{ap},\varepsilon}(T) = \left\{ z\in\mathbb C: \exists (u_n)\subset \Y,\ \|u_n\|_\Y=1,\ \limsup_{n\to\infty}\|(T-zI)u_n\|_\Y\le\varepsilon \right\}.
\]

\item[iii)] One has
\[
        R_{\mathrm{ap},\varepsilon}(T) \subseteq C_{\mathrm{ap},\varepsilon}(T).
\]

\item[iv)]
\[
        C_{\mathrm{ap},\varepsilon}(T) = \bigcap_{k=1}^{\infty} R_{\mathrm{ap},\varepsilon+2^{-k}}(T).
\]

\item[v)]
\[
        \sigma_{\mathrm{ap}}(T) = \bigcap_{k=1}^{\infty}R_{\mathrm{ap},2^{-k}}(T) = \bigcap_{k=1}^{\infty}C_{\mathrm{ap},2^{-k}}(T).
\]

\item[vi)] If \(\varepsilon>0\), then \(R_{\mathrm{ap},\varepsilon}(T)\), \(C_{\mathrm{ap},\varepsilon}(T)\), and \(\sigma_{\mathrm{ap}}(T)\) are compact non-empty subsets of \(\mathbb C\).
\end{enumerate}
\end{lemma}

\begin{proof}
For \(z,w\in\mathbb C\) and \(\|u\|_\Y=1\),
\[
        \|(T-zI)u\|_\Y \le \|(T-wI)u\|_\Y +|z-w|.
\]
Taking the infimum over all unit vectors \(u\) gives
\[
        \nu_T(z)\le \nu_T(w)+|z-w|.
\]
Interchanging \(z\) and \(w\) proves the Lipschitz estimate.

For the sequence characterization, suppose first that \(z\in C_{\mathrm{ap},\varepsilon}(T)\). Then \(\nu_T(z)\le\varepsilon\). For each \(n\in\mathbb N\), choose \(u_n\in \Y\) with \(\|u_n\|=1\) and
\[
        \|(T-zI)u_n\|_\Y \le \nu_T(z)+\frac1n.
\]
Then
\[
        \limsup_{n\to\infty}\|(T-zI)u_n\|_\Y \le \nu_T(z) \le \varepsilon.
\]
Conversely, if such a sequence exists, then
\[
        \nu_T(z) \le \|(T-zI)u_n\|_\Y,
\]
and hence
\[
        \nu_T(z) \le \limsup_{n\to\infty}\|(T-zI)u_n\|_\Y \le \varepsilon.
\]
Thus \(z\in C_{\mathrm{ap},\varepsilon}(T)\).

The inclusion
\[
        R_{\mathrm{ap},\varepsilon}(T) \subseteq C_{\mathrm{ap},\varepsilon}(T)
\]
follows from continuity of \(\nu_T\).

We prove now the intersection identity. If \(z\in C_{\mathrm{ap},\varepsilon}(T)\), then
\[
        \nu_T(z)\le \varepsilon<\varepsilon+2^{-k}
\]
for every \(k\), i.e.
\[
        z\in \{w:\nu_T(w)<\varepsilon+2^{-k}\} \subseteq R_{\mathrm{ap},\varepsilon+2^{-k}}(T)
\]
for every \(k\). Conversely, suppose
\[
        z\in\bigcap_{k=1}^{\infty}R_{\mathrm{ap},\varepsilon+2^{-k}}(T).
\]
Since
\[
        R_{\mathrm{ap},\varepsilon+2^{-k}}(T) \subseteq \{w:\nu_T(w)\le \varepsilon+2^{-k}\},
\]
we get
\[
        \nu_T(z)\le \varepsilon+2^{-k}.
\]
Letting \(k\to\infty\) gives \(\nu_T(z)\le\varepsilon\), hence \(z\in C_{\mathrm{ap},\varepsilon}(T)\).

The identities for \(\sigma_{\mathrm{ap}}(T)\) follow by applying the previous argument to the zero level
\[
        \sigma_{\mathrm{ap}}(T)=\{z:\nu_T(z)=0\}.
\]

Finally,
\[
        \nu_T(z) = \inf_{\|u\|=1}\|(T-zI)u\| \ge \big||z|-\|T\|\big| \ge |z|-\|T\|.
\]
Thus every set \(\{z:\nu_T(z)\le\varepsilon\}\) is contained in \(\overline B(0,\|T\|+\varepsilon)\). Since \(\nu_T\) is continuous, the closed sublevel set \(C_{\mathrm{ap},\varepsilon}(T)\) is compact, and
\(R_{\mathrm{ap},\varepsilon}(T)\subseteq C_{\mathrm{ap},\varepsilon}(T)\) is closed by definition, hence compact. The approximate point spectrum is compact because it is the closed zero set of \(\nu_T\), contained in \(\sigma(T)\).

Non-emptiness follows from the standard spectral fact
\[
        \partial\sigma(T)\subseteq\sigma_{\mathrm{ap}}(T),
\]
for bounded operators on complex Banach spaces; see, e.g., \cite[Theorem~1.1.7(iii)]{Bourhim2000}. Since the spectrum of a bounded operator on a non-zero complex Banach space is non-empty and compact, its boundary is non-empty. Hence \(\sigma_{\mathrm{ap}}(T)\neq\emptyset\), and therefore also \(R_{\mathrm{ap},\varepsilon}(T)\) and \(C_{\mathrm{ap},\varepsilon}(T)\) are non-empty for \(\varepsilon>0\).
\end{proof}

The inclusion
\[
        R_{\mathrm{ap},\varepsilon}(T) \subseteq C_{\mathrm{ap},\varepsilon}(T)
\]
can be strict for bounded operators. Thus the two fixed-\(\varepsilon\) targets are not merely notational variants. On the other hand, the next result shows that this distinction disappears for measure-preserving Koopman operators.

\begin{remark}[Isometric but not necessarily invertible]
If \(F_\#\omega=\omega\), then \(\mathcal K_F\) is an isometry on \(L^p(\X,\omega)\). It need not be surjective unless \(F\) is invertible modulo \(\omega\). Thus both surjective and non-surjective isometries occur naturally for Koopman operators. The proof of \cref{thm:isometry-closed-regularized} treats these two cases separately.
\end{remark}

\begin{theorem}[Closed and regularized fixed-\(\varepsilon\) targets for \(L^p\)-isometries]\label{thm:isometry-closed-regularized}
Let \(1<p<\infty\), and let \(V\) be an isometry on a complex \(L^p\)-space. Then, for every \(\varepsilon>0\),
\[
        C_{\mathrm{ap},\varepsilon}(V) = R_{\mathrm{ap},\varepsilon}(V).
\]
Consequently, if \(F_\#\omega=\omega\), then
\[
        C_{\mathrm{ap},\varepsilon}(\K_F) = R_{\mathrm{ap},\varepsilon}(\K_F).
\]
\end{theorem}

\begin{proof}
First suppose that \(V\) is not surjective. By the standard spectral theorem for Banach-space isometries, see, e.g., \cite[Lemma~1.1]{IlisevicTurnsekIsometries} we get $\sigma(V)=\overline{\mathbb D}$. Moreover,
\[
        \partial\sigma(V)\subseteq\sigma_{\mathrm{ap}}(V),
\]
therefore
\[
        \mathbb T\subseteq\sigma_{\mathrm{ap}}(V).
\]
For every unit vector \(u\) we have
\[
        \|(V-zI)u\| \ge \bigl|\|Vu\|-|z|\|u\|\bigr| = |1-|z||.
\]
Taking the infimum over all unit vectors gives then
\[
        \nu_V(z)\ge |1-|z||.
\]
Conversely, if \(z\neq0\), put
\[
        \zeta:=\frac{z}{|z|}\in\mathbb T.
\]
Since \(\zeta\in\sigma_{\mathrm{ap}}(V)\), there are unit vectors \(u_n\) such that
\[
        \|(V-\zeta I)u_n\|\to0.
\]
Therefore
\[
        \|(V-zI)u_n\| \le \|(V-\zeta I)u_n\|+|\zeta-z|.
\]
Taking \(n\to\infty\) gives
\[
        \nu_V(z)\le |\zeta-z|=|1-|z||.
\]
For \(z=0\), the equality is immediate from \(\|Vu\|=\|u\|\). Hence
\[
        \nu_V(z)=\operatorname{dist}(z,\mathbb T),
\]
where $z\in\mathbb C$. In other words,
\[
        C_{\mathrm{ap},\varepsilon}(V) = R_{\mathrm{ap},\varepsilon}(V) = \{z:\operatorname{dist}(z,\mathbb T)\le\varepsilon\}.
\]

Assume now that \(V\) is surjective. Then \(V^{-1}\) is also an isometry. The standard Neumann-series argument gives
\[
        \sigma(V)\subseteq\mathbb T.
\]
Indeed, if \(|z|>1\), then \(zI-V\) is invertible by
\[
        (zI-V)^{-1} = z^{-1}\sum_{n=0}^{\infty}z^{-n}V^n,
\]
and if \(|z|<1\), then
\[
        V-zI= V(I-zV^{-1}),
\]
where \(I-zV^{-1}\) is invertible by the Neumann-series. Thus no spectral point lies outside \(\mathbb T\).

Every spectral point of \(V\) is therefore a boundary spectral point, hence belongs to \(\sigma_{\mathrm{ap}}(V)\). Thus
\[
        \nu_V(z)=0
\]
for $z\in\sigma(V)$. Let \(z\in\rho(V)\), the resolvent set of \(V\). Since \(V-zI\) is invertible,
\[
        \nu_V(z) = \|(V-zI)^{-1}\|^{-1}.
\]
Suppose now, for contradiction, that
\[
        C_{\mathrm{ap},\varepsilon}(V) \neq R_{\mathrm{ap},\varepsilon}(V).
\]
Then there is a
\[
        z_0\in C_{\mathrm{ap},\varepsilon}(V) \setminus R_{\mathrm{ap},\varepsilon}(V).
\]
Since \(R_{\mathrm{ap},\varepsilon}(V)\) is the closure of \(\{\nu_V<\varepsilon\}\), this implies
\[
        \nu_V(z_0)=\varepsilon
\]
and there exists an open neighbourhood \(U\) of \(z_0\) such that
\[
        \nu_V(z)\ge\varepsilon
\]
for $z\in U$. Because \(\varepsilon>0\), the point \(z_0\) cannot lie in \(\sigma(V)\); hence, after shrinking \(U\) if necessary, we may assume
\[
        U\subset\rho(V).
\]
On \(U\), the resolvent norm satisfies
\[
        \|(V-zI)^{-1}\| = \nu_V(z)^{-1} \le \varepsilon^{-1} = \|(V-z_0I)^{-1}\|.
\]
Thus the resolvent norm has a local maximum at \(z_0\).

The function
\[
        z\mapsto (V-zI)^{-1}
\]
is operator-valued holomorphic on \(\rho(V)\), and the norm of a Banach-valued holomorphic function is subharmonic. Since $L^p(\mathcal X,\omega)$ is uniformly convex (and therefore also complex uniformly convex) for $1<p<\infty$, we have by the strong maximum property for resolvents in \cite[Theorem~1.2]{ChandlerWildeChonchaiyaLindner2012}
\[
        \|(V-zI)^{-1}\|<\varepsilon^{-1} \qquad \forall z\in U
\]
so in particular for $z=z_0$. But \(\nu_V(z_0)=\varepsilon\), and since \(z_0\in\rho(V)\),
\[
        \|(V-z_0I)^{-1}\| = \nu_V(z_0)^{-1} = \varepsilon^{-1}
\]
a contradiction. Thus no such \(z_0\) exists, and
\[
        C_{\mathrm{ap},\varepsilon}(V) = R_{\mathrm{ap},\varepsilon}(V).
\]

Finally, if \(F_\#\omega=\omega\), then \cref{prop:boundednessLp} shows that \(\K_F\) is an isometry on \(L^p(\X,\omega)\). Applying the preceding result to \(V=\K_F\) proves the claim.
\end{proof}

\section{Residual Towers In \texorpdfstring{$L^{p}(\mathcal{X},\omega)$}{Lp(X,w)}}\label{sec:residual-towers}

The upper-bound construction uses actual residuals
\[
        \|(\K_F-zI)g\|_p
\]
on finite-dimensional trial spaces. The trial functions are continuous, so for continuous \(F\) the integrand
\[
        x\mapsto |g(F(x))-zg(x)|^p
\]
is continuous and can be approximated by tagged quadrature.

\subsection{Finite-Dimensional Residuals}\label{sec:finite_residual}

\begin{lemma}[Dense independent Lipschitz dictionary]\label{lem:dense-lip-dictionary}
There exists a nondecreasing sequence of finite-dimensional subspaces
\[
        V_1\subseteq V_2\subseteq\cdots\subset C(\X)\cap\Lip(\X)
\]
such that
\[
        \overline{\bigcup_{n\ge1}V_n}^{\,L^p(\X,\omega)} = L^p(\X,\omega).
\]
Moreover, each \(V_n\) may be equipped with a fixed basis
\[
        \psi_{n,1},\ldots,\psi_{n,d_n}.
\]
If \(L^p(\X,\omega)\) is infinite-dimensional, the \(V_n\) may be chosen as
\[
        V_n=\operatorname{span}\{\psi_1,\ldots,\psi_n\}
\]
for a countably infinite linearly independent family
\[
        (\psi_j)_{j\ge1}\subset C(\X)\cap\Lip(\X).
\]
\end{lemma}

\begin{proof}
First, \(C(\X)\) is dense in \(L^p(\X,\omega)\), because \(\omega\) is a finite Borel measure on the compact metric space \(\X\).

Second, Lipschitz functions are uniformly dense in \(C(\X)\). A possible proof of this is the standard McShane-Whitney/Lipschitz-envelope approximation, see, e.g., \cite[Theorem~1.33]{weaverLip}.

Since \(\X\) is compact metric, \(C(\X)\) is separable in the supremum norm. By the preceding Lipschitz-density argument, \(C(\X)\cap\Lip(\X)\) is uniformly dense in \(C(\X)\), and therefore also dense in \(L^p(\X,\omega)\). Choose a sequence
\[
        (\varphi_j)_{j\ge1}\subset C(\X)\cap\Lip(\X)
\]
whose complex linear span is dense in \(L^p(\X,\omega)\). We now extract a linearly independent subsequence without changing the algebraic span generated at any finite stage. Define finite ordered sets
\[
        S_j\subset\{\varphi_1,\ldots,\varphi_j\}
\]
for $j\ge 0$ inductively as follows. Set $S_0:=\emptyset$. Assume \(S_{j-1}\) has been defined. If
\[
        \varphi_j\in\operatorname{span}_{\mathbb C} S_{j-1},
\]
set $S_j:=S_{j-1}$. If
\[
        \varphi_j\notin\operatorname{span}_{\mathbb C} S_{j-1},
\]
set $S_j:=S_{j-1}\cup\{\varphi_j\}$, where \(\varphi_j\) is appended at the end of the ordered list.

We claim now that for every \(j\ge 0\),
\[
        \operatorname{span}_{\mathbb C} S_j = \operatorname{span}_{\mathbb C}\{\varphi_1,\ldots,\varphi_j\}.
\]
The claim is obviously true for \(j=0\). Suppose it holds for \(j-1\). If
\[
        \varphi_j\in\operatorname{span}_{\mathbb C}S_{j-1},
\]
then
\[
\begin{aligned}
        \operatorname{span}_{\mathbb C}\{\varphi_1,\ldots,\varphi_j\}
        &=\operatorname{span}_{\mathbb C}\{\varphi_1,\ldots,\varphi_{j-1}\} \\
        &=\operatorname{span}_{\mathbb C}S_{j-1}=
        \operatorname{span}_{\mathbb C}S_j .
\end{aligned}
\]
If
\[
        \varphi_j\notin\operatorname{span}_{\mathbb C}S_{j-1},
\]
then
\[
\begin{aligned}
        \operatorname{span}_{\mathbb C}S_j
        &=\operatorname{span}_{\mathbb C}\bigl(S_{j-1}\cup\{\varphi_j\}\bigr) \\
        &=\operatorname{span}_{\mathbb C}\{\varphi_1,\ldots,\varphi_{j-1},\varphi_j\},
\end{aligned}
\]
again by the induction hypothesis. This proves the claim.

Now let $(\psi_k)_{k\in I}$ be the retained functions, in their order of retention, where either $I=\{1,\ldots,N\}$ for some \(N\in\mathbb N\), or $I=\mathbb N$. By construction, \((\psi_k)_{k\in I}\) is linearly independent. Moreover,
\[
        \operatorname{span}_{\mathbb C}\{\psi_k:k\in I\} =
        \operatorname{span}_{\mathbb C}\{\varphi_j:j\ge1\}.
\]
Thus the two algebraic spans coincide. Since the span of \((\varphi_j)\) is dense in \(L^p(\X,\omega)\), the span of the retained family \((\psi_k)\) is also dense in \(L^p(\X,\omega)\).

If \(I=\mathbb N\), set
\[
        V_n:=\operatorname{span}_{\mathbb C}\{\psi_1,\ldots,\psi_n\}.
\]
If \(I=\{1,\ldots,N\}\) is finite, set
\[
        V_n:= \operatorname{span}_{\mathbb C}\{\psi_1,\ldots,\psi_{\min(n,N)}\}.
\]
In both cases,
\[
        V_1\subseteq V_2\subseteq\cdots
\]
and
\[
        \overline{\bigcup_{n\ge1}V_n}^{\,L^p(\X,\omega)} = L^p(\X,\omega).
\]
W.l.o.g. we choose the dictionary so that $\psi_1 \equiv 1$ (consequently $\mathbf{1}_\X \in V_n$ for every $n$.).
\end{proof}

For notational simplicity, fix for each $n$ a basis $\psi_{n,1},\ldots,\psi_{n,d_n}$ of $V_n$, where $d_n = \dim V_n$, and write
\[
        \Psi_n c:=\sum_{j=1}^{d_n} c_j\psi_{n,j},
\]
where $c \in \mathbb{C}^{d_n}$. Since $\mathbf 1_{\mathcal X}\in V_n$ for every $n$, we choose the bases so that
\[
   \psi_{n,1}=\mathbf 1_{\mathcal X}
   \qquad\forall n.
\]

\begin{definition}[Finite-dimensional lower residual]\label{def:finite-residual}
For \(F\in\Omega_\X\), \(z\in\mathbb C\), and \(n\in\mathbb N\), define
\[
        r_n(z,F) := \min_{\|c\|_2=1} \frac{\|(\K_F-zI)\Psi_n c\|_p}{\|\Psi_n c\|_p}.
\]
Equivalently,
\[
        r_n(z,F) = \inf_{\substack{g\in V_n\\ \|g\|_p=1}} \|(\K_F-zI)g\|_p.
\]
\end{definition}

\begin{proposition}[Finite-dimensional residual convergence]\label{prop:finite-residual-convergence}
For every \(F\in\Omega_\X\) and every \(z\in\mathbb C\),
\[
        r_n(z,F)\downarrow \nu_{\K_F}(z).
\]
Moreover, the convergence is uniform on every compact disk
\[
        \overline B_R(0):=\{z\in\mathbb C:|z|\le R\}.
\]
\end{proposition}

\begin{proof}
The spaces \(V_n\) are nested by choice. Hence
\[
        r_{n+1}(z,F)\le r_n(z,F).
\]
Also,
\[
        r_n(z,F)\ge \nu_{\K_F}(z)
\]
for all \(n\). Now let \(\eta>0\). By definition of the lower norm, choose \(g\in L^p(\X,\omega)\) with $\|g\|_p=1$ and
\[
        \|(\K_F-zI)g\|_p<\nu_{\K_F}(z)+\frac{\eta}{2}.
\]
Since \(\bigcup_n V_n\) is dense in \(L^p(\X,\omega)\), choose \(N\) and \(v\in V_N\) such that
\[
        \|v-g\|_p<\delta,
\]
where \(0<\delta<1/2\) will be chosen below. Then \(v\neq0\), and
\[
        \left\|\frac{v}{\|v\|_p}-g\right\|_p \le 2\|v-g\|_p < 2\delta.
\]
Indeed,
\[
        \left\|\frac{v}{\|v\|_p}-g\right\|_p \le \left\|\frac{v}{\|v\|_p}-v\right\|_p+\|v-g\|_p = |1-\|v\|_p|+\|v-g\|_p \le 2\|v-g\|_p.
\]
Now set
\[
        \widehat v:=\frac{v}{\|v\|_p}.
\]
Then
\[
        \|(\K_F-zI)\widehat v\|_p \le \|(\K_F-zI)g\|_p + \|(\K_F-zI)(\widehat v-g)\|_p \le \nu_{\K_F}(z)+\frac{\eta}{2} + (\|\K_F\|+|z|)\,2\delta .
\]
Choose therefore
\[
        \delta<\frac{\eta}{4(\|\K_F\|+|z|+1)}
\]
to get
\[
        r_N(z,F) \le \|(\K_F-zI)\widehat v\|_p < \nu_{\K_F}(z)+\eta.
\]
Since \(\eta>0\) was arbitrary,
\[
        \lim_{n\to\infty}r_n(z,F)=\nu_{\K_F}(z).
\]

It remains to prove uniform convergence on compact disks. For \(z,w\in\mathbb C\) and \(g\in V_n\) with \(\|g\|_p=1\),
\[
        \|(\K_F-zI)g\|_p \le \|(\K_F-wI)g\|_p+|z-w|.
\]
Taking infima over \(g\in V_n\), \(\|g\|_p=1\), and then interchanging \(z,w\), gives
\[
        |r_n(z,F)-r_n(w,F)|\le |z-w|.
\]
Likewise,
\[
        |\nu_{\K_F}(z)-\nu_{\K_F}(w)|\le |z-w|.
\]
Thus \(r_n(\cdot,F)\) and \(\nu_{\K_F}\) are continuous on every compact disk. Since \(r_n(\cdot,F)\downarrow\nu_{\K_F}(\cdot)\) pointwise, Dini's theorem gives uniform convergence on every \(\overline B_R(0)\).
\end{proof}

\subsection{Tagged Quadrature Residuals}

We now turn the residuals \(r_n\) into finite-query base algorithms. The only values of \(F\) used at quadrature level \(m\) are
\[
        F(x_{m,P}), \qquad P\in\mathcal P_m,
\]
where \(x_{m,P}\) is a tag in the cell \(P\).

Choose finite Borel partitions $\mathcal P_m=\{P\}$ with
\[
        \operatorname{mesh}(\mathcal P_m):= \max_{P\in\mathcal P_m}\sup_{x,y\in P}d_\X(x,y)\longrightarrow 0,
\]
and choose tags \(x_{m,P}\in P\). For every \(h\in C(\X)\),
\[
        \sum_{P\in\mathcal P_m} h(x_{m,P})\omega(P) \longrightarrow \int_\X h\,d\omega .
\]
Indeed,
\[
\begin{aligned}
\left| \sum_{P\in\mathcal P_m}h(x_{m,P})\omega(P)-\int_\X h\,d\omega \right| &\le \sum_{P\in\mathcal P_m}\int_P |h(x_{m,P})-h(x)|\,d\omega(x) \le \\
&\omega(\X)\sup_{\substack{x,y\in \X\\d_\X(x,y)\le \operatorname{mesh}(\mathcal P_m)}} |h(x)-h(y)|,
\end{aligned}
\]
which tends to \(0\) by uniform continuity of \(h\).

\begin{definition}[Tagged quadrature residuals]\label{def:tagged-quadrature-residuals}
For \(h\in C(\X)\), set
\[
        Q_m(h)^p := \sum_{P\in\mathcal P_m}|h(x_{m,P})|^p \omega(P).
\]
For \(z\in\mathbb C\), define
\[
        r_{n,m}(z,F) := \min_{\|c\|_2=1} \frac{Q_m\big((\Psi_n c)\circ F-z\Psi_n c\big)}{\|\Psi_n c\|_p}.
\]
\end{definition}

\begin{lemma}[Finite-query property]\label{lem:finite-query-property}
For fixed \(n,m\), and \(z\), the value \(r_{n,m}(z,F)\) depends on \(F\) only through the finite table
\[
        \{F(x_{m,P}): P\in\mathcal P_m\}.
\]
Consequently every grid set formed from finitely many residuals \(r_{n,m}(z,F)\) is a general algorithm in the sense of \cref{def:Gen_alg}.
\end{lemma}

\begin{proof}
For fixed \(n,m,z\) and \(c\in\mathbb C^{d_n}\),
\[
        Q_m\big((\Psi_n c)\circ F-z\Psi_n c\big)^p = \sum_{P\in\mathcal P_m} \left| \sum_{j=1}^n c_j\psi_j(F(x_{m,P})) - z\sum_{j=1}^n c_j\psi_j(x_{m,P}) \right|^p \omega(P).
\]
All quantities except the finitely many complex valued \(\psi_{n,j}(F(x_{m,P}))\) are fixed problem points. Hence $r_{n,m}(z,F)$ depends only on finitely many allowed evaluations from $\lambda_\X$.

Since \(c\mapsto \|\Psi_n c\|_p\) has a strictly positive minimum on \(\{c \in \mathbb C^{d_n} :\|c\|_2=1\}\), the quotient is a continuous function of \(c\) on a compact set. Hence the minimum exists and is determined by the same finite collection of oracle values. The consistency axiom for general algorithms follows immediately: if two inputs agree on these queried oracle values, then the entire finite table, all residuals formed from it, and hence the resulting grid set, agree.
\end{proof}

\begin{proposition}[Quadrature convergence]\label{prop:quadrature-convergence}
Fix \(F\in\Omega_\X\), \(n\in\mathbb N\), and \(R>0\). Then
\[
        r_{n,m}(\cdot,F)\to r_n(\cdot,F)
\]
uniformly on \(\overline B_R(0)\). If \(F\in\Omega_\X^\alpha\), then the convergence depends only on \(n,R,\alpha\), a tolerance $\eta > 0$, and the dictionary functions, uniformly in \(F\).
\end{proposition}

\begin{proof}
Let
\[
        S_n:=\{c\in\mathbb C^{d_n} : \|c\|_2=1\}.
\]
Because $\psi_{n,1},\ldots,\psi_{n,d_n}$ are linearly independent in $L^p(\mathcal X,\omega)$, the continuous function $c\mapsto \|\Psi_n c\|_p$ has a strictly positive minimum on \(S_n\). Denote it by
\[
        \mu_n:=\min_{c\in S_n}\|\Psi_n c\|_p>0.
\]

For fixed \(F,n,R\), define
\[
        h_{c,z,F}(x) := |\Psi_n c(F(x))-z\Psi_n c(x)|^p,
\]
for $c\in S_n,\ |z|\le R$. Now the map
\[
        (c,z,x)\mapsto h_{c,z,F}(x)
\]
is continuous on the compact set
\[
        S_n\times\overline B_R(0)\times \X.
\]
Therefore the family
\[
        \mathcal H_{n,R,F}:= \{h_{c,z,F}:c\in S_n,\ |z|\le R\}
\]
is uniformly bounded and equicontinuous.

For every \(h\in\mathcal H_{n,R,F}\),
\[
        \left| \sum_{P\in\mathcal P_m}h(x_{m,P})\omega(P)-\int_\X h\,d\omega \right| \le \omega(\X) \sup_{\substack{x,y\in \X\\d_\X(x,y)\le\operatorname{mesh}(\mathcal P_m)}} |h(x)-h(y)|.
\]
Taking the supremum over \(h\in\mathcal H_{n,R,F}\), the right-hand side tends to \(0\) by uniform equicontinuity. Hence
\[
        \sup_{\substack{c\in S_n\\ |z|\le R}} \left| Q_m\big((\Psi_n c)\circ F-z\Psi_n c\big)^p - \|(\Psi_n c)\circ F-z\Psi_n c\|_p^p \right| \to 0.
\]
Since
\[
        |a^{1/p}-b^{1/p}|\le |a-b|^{1/p}
\]
for $a,b\ge 0$, we also have uniform convergence of the \(p\)-th roots. Dividing by the lower bound \(\mu_n>0\) gives
\[
        \sup_{|z|\le R}|r_{n,m}(z,F)-r_n(z,F)|\to 0.
\]

Now suppose \(F\in\Omega_\X^\alpha\). For \(c\in S_n\), the finite-dimensional equivalence of norms gives a constant \(C_n\) such that
\[
        \|\Psi_n c\|_\infty\le C_n, \quad \Lip(\Psi_n c)\le C_n,
\]
where $c\in S_n$. For \(x,y\in \X\) and \(|z|\le R\),
\[
        |\Psi_n c(F(x))-z\Psi_n c(x)-\Psi_n c(F(y))+z\Psi_n c(y)| \le C_n\alpha(d_\X(x,y))+R C_n d_\X(x,y).
\]
Together with the uniform bound on
\[
        |\Psi_n c(F(x))-z\Psi_n c(x)|,
\]
the elementary inequality
\[
        \big||a|^p-|b|^p\big| \le p\max\{|a|,|b|\}^{p-1}|a-b|
\]
shows that the family \(\mathcal H_{n,R,F}\) has a modulus of continuity bounded in terms of \(n,R,\alpha\), a tolerance $\eta > 0$ and the dictionary functions only (meaning: we get $\sup_{|z|\le R}|r_{n,m}(z,F)-r_n(z,F)|<\eta$ for $m\ge m_0$, where $m_0=m_0(n,R,\alpha,\eta,\Psi)$). Therefore one can choose \(m\) as a function of \(n,R,\alpha\), and the dictionary functions, uniformly for all \(F\in\Omega_\X^\alpha\).
\end{proof}

A minor technical point is needed because the residual estimates \(r_{n,m}(z,F)\) need not converge monotonically in \(m\).  Directly testing the fixed threshold \(\varepsilon\) may therefore fail to stabilize at grid points whose limiting residual is exactly \(\varepsilon\). The following elementary finite-grid lemma chooses, from finitely many thresholds slightly below \(\varepsilon\), one that avoids the limiting residual values on the grid. This gives an eventually constant inner limit without assuming monotonicity.

\begin{lemma}[Finite-grid threshold selection with finite mind changes]\label{lem:threshold-selection}
Let
\[
        G=\{z_1,\ldots,z_N\}\subset\mathbb C
\]
be finite and
\[
        a_i^{(m)}\to a_i, \qquad i=1,\ldots,N
\]
be convergent real sequences. Fix \(0<\tau<\varepsilon\), and define
\[
        t_j:=\varepsilon-\frac{j\tau}{N+2}, \qquad j=1,\ldots,N+1.
\]
Then there exists a sequence \(j_m\in\{1,\ldots,N+1\}\), where \(j_m\) is determined by
\[
        \{a_i^{(s)}:1\le i\le N,\ 1\le s\le m\},
\]
such that \(j_m\) is eventually constant. If \(j_*\) denotes its eventual value, then
\[
        t_{j_*}\notin\{a_1,\ldots,a_N\}.
\]
Consequently the sets
\[
        \Gamma_m:=\{z_i\in G : a_i^{(m)}<t_{j_m}\}
\]
are eventually constant, with eventual value
\[
        \Gamma:=\{z_i\in G:a_i<t_{j_*}\}.
\]
\end{lemma}

\begin{proof}
For \(j=1,\ldots,N+1\), define
\[
        d_j^{(m)}:=\min_{1\le i\le N}|a_i^{(m)}-t_j|, \qquad d_j:=\min_{1\le i\le N}|a_i-t_j|.
\]
Then
\[
        d_j^{(m)}\to d_j, \qquad j=1,\ldots,N+1.
\]
Because there are \(N\) limiting values \(a_1,\ldots,a_N\) and \(N+1\) thresholds \(t_1,\ldots,t_{N+1}\), at least one threshold avoids all limiting values. Thus there exists \(j\) with $d_j>0$. We now define \(j_m\): For \(q,j\in\mathbb N\) with
\[
        q\le m, \qquad 1\le j\le N+1,
\]
say that the pair \((q,j)\) is certified at stage \(m\) if
\[
        d_j^{(s)}>2^{-q}
\]
for every $s=q,q+1,\ldots,m$. Notice that, for a fixed pair $(q,j)$, certification is monotone in the stage $m$: if $(q,j)$ is not certified at some stage $m\ge q$, then it is not certified at any later stage. Indeed, the same violating index $s\in\{q,\ldots,m\}$ remains among the indices tested at every later stage. If at least one pair is certified at stage \(m\), let \((q_m,j_m)\) be the lexicographically smallest certified pair. If no pair is certified, set \(j_m:=1\).

We now prove that \(j_m\) is eventually constant. Since some \(d_j>0\), choose \(j^\sharp\) with \(d_{j^\sharp}>0\). Choose \(q^\sharp\) so large that
\[
        2^{-q^\sharp}<\frac{d_{j^\sharp}}{2}.
\]
Since \(d_{j^\sharp}^{(s)}\to d_{j^\sharp}\), increasing \(q^\sharp\) if necessary we may assume
\[
        d_{j^\sharp}^{(s)}>2^{-q^\sharp}
\]
for all $s\ge q^\sharp$. Hence \((q^\sharp,j^\sharp)\) is certified at every stage \(m\ge q^\sharp\). Thus certified pairs exist for all large \(m\).

Let \((q_*,j_*)\) be the lexicographically smallest pair with the property that it is certified at every sufficiently large stage. Such a pair exists by the previous argument. Any lexicographically smaller pair is not certified eventually; hence for each such pair there exists a stage after which it is never certified. Since there are only finitely many lexicographically smaller pairs than \((q_*,j_*)\),
there exists \(M\) such that for all \(m\ge M\), no smaller pair is certified and \((q_*,j_*)\) is certified. Therefore
\[
        j_m=j_*
\]
for $m\ge M$.

Since \((q_*,j_*)\) is certified for arbitrarily large \(m\), for all sufficiently large \(s\),
\[
        d_{j_*}^{(s)}>2^{-q_*}.
\]
Passing to the limit gives
\[
        d_{j_*}\ge 2^{-q_*}>0.
\]
Therefore
\[
        t_{j_*}\notin\{a_1,\ldots,a_N\}.
\]
Because \(t_{j_*}\) is separated from every \(a_i\), the inequalities
\[
        a_i^{(m)}<t_{j_*}
\]
stabilize for each \(i\). Since also \(j_m=j_*\) eventually, the sets \(\Gamma_m\) are eventually constant with the stated value.
\end{proof}

We can now prove the no-modulus upper bound for the regularized target. The inner limit is the quadrature limit; the outer limit is the finite-dimensional residual limit. We will sometimes write $\nu_F(z) := \nu_{\K_F}(z)$ for readability.

\begin{theorem}[No-modulus upper bound for the regularized target]\label{thm:R-upper-nomodulus}
Let \(1<p<\infty\) and \(\varepsilon>0\). Then
\[
        R_{\mathrm{ap},\varepsilon}(\K_F)\in\Sigma_2^G \quad\text{on }\Omega_\X.
\]
Consequently,
\[
        R_{\mathrm{ap},\varepsilon}(\K_F)\in\Sigma_2^G \quad\text{on }\Omega_\X^m.
\]
\end{theorem}

\begin{proof}
Let
\[
        G_n:= \left\{ \frac{k+i\ell}{n}: k,\ell\in\mathbb Z,\ \left|\frac{k+i\ell}{n}\right|\le n \right\} \cup\{1\}.
\]
The point \(1\) is included because
\[
        \K_F \mathbf 1_\X = \mathbf 1_\X,
\]
or in other words: \(1\in\sigma_{\mathrm{ap}}(\K_F) \subset R_{\mathrm{ap},\varepsilon}(\K_F)\).

Let
\[
        \tau_n:=\min\{\varepsilon/2,n^{-1/2}\}.
\]
Enumerate
\[
        G_n=\{z_{n,1},\ldots,z_{n,N_n}\}.
\]
For fixed \(n\), define further
\[
        a_i^{(m)}:=r_{n,m}(z_{n,i},F), \qquad a_i:=r_n(z_{n,i},F).
\]
By \cref{prop:quadrature-convergence},
\[
        a_i^{(m)}\to a_i, \qquad i=1,\ldots,N_n.
\]
Apply \cref{lem:threshold-selection} with
\[
        G=G_n, \quad \tau=\tau_n.
\]
This produces thresholds
\[
        t_{n,j}:=\varepsilon-\frac{j\tau_n}{N_n+2}, \qquad j=1,\ldots,N_n+1,
\]
and an index \(j_{n,m}\) determined by
\[
        \{r_{n,s}(z,F) : z\in G_n,\ 1\le s\le m\},
\]
such that \(j_{n,m}\) is eventually constant. Define
\[
        \Gamma_{n,m}(F) := \{z\in G_n:r_{n,m}(z,F)<t_{n,j_{n,m}}\}.
\]
Importantly, the sets $\Gamma_{n,m}(F)$ are non-empty: Since $1 \in G_n$ and $\mathbf 1_\X \in V_n$ we have $r_{n,m}(1,F)=0$ because $(\K_F - I)\mathbf 1_\X = 0$ pointwise. Since all $t_{n,j}$ are positive, $1 \in \Gamma_{n,m}(F)$. We claim that each \(\Gamma_{n,m}\) is a general algorithm and indeed: for fixed \(n,m\), it uses only the finite union of point-evaluation tables
\[
        \{F(x_{s,P}):P\in\mathcal P_s,\ 1\le s\le m\}.
\]
By \cref{lem:threshold-selection}, the inner limit
\[
        \Gamma_n(F):=\lim_{m\to\infty}\Gamma_{n,m}(F)
\]
exists and equals
\[
        \Gamma_n(F) = \{z\in G_n:r_n(z,F)<t_n(F)\},
\]
where
\[
        t_n(F)\in(\varepsilon-\tau_n,\varepsilon)
\]
and
\[
        t_n(F)\notin\{r_n(z,F) : z\in G_n\}.
\]

We prove now
\[
        \Gamma_n(F)\to R_{\mathrm{ap},\varepsilon}(\K_F)
\]
in the Hausdorff metric. First, if \(z\in\Gamma_n(F)\), then
\[
        r_n(z,F)<t_n(F)<\varepsilon.
\]
Since
\[
        \nu_F(z)\le r_n(z,F),
\]
we have
\[
        \nu_F(z)<\varepsilon.
\]
Thus
\[
        \Gamma_n(F)\subseteq R_{\mathrm{ap},\varepsilon}(\K_F).
\]

Conversely, let
\[
        K:=R_{\mathrm{ap},\varepsilon}(\K_F).
\]
By \cref{lem:basic-lower-norm-relations}, \(K\) is compact. Fix \(z\in K\) and \(\eta>0\). Since \(K\) is the closure of \(\{\nu_F<\varepsilon\}\), choose \(w\in\mathbb C\) such that
\[
        |z-w|<\eta/3, \quad \nu_F(w)<\varepsilon.
\]
Choose further \(\delta>0\) such that $\delta < \eta /3$ and
\[
        \nu_F(w)+4\delta<\varepsilon.
\]
Then for all sufficiently large \(n\), we get
\[
        \tau_n<\delta,
\]
there exists \(w_n\in G_n\) with
\[
        |w_n-w|<\delta,
\]
and, by \cref{prop:finite-residual-convergence},
\[
        r_n(w_n,F)<\nu_F(w_n)+\delta.
\]
Using the \(1\)-Lipschitz continuity of \(\nu_F\),
\[
        \nu_F(w_n)\le \nu_F(w)+|w_n-w|<\nu_F(w)+\delta.
\]
Therefore
\[
        r_n(w_n,F)<\nu_F(w)+2\delta.
\]
Since
\[
        \nu_F(w)+4\delta<\varepsilon \quad\text{and}\quad \tau_n<\delta,
\]
we obtain
\[
        r_n(w_n,F) < \varepsilon-\tau_n < t_n(F).
\]
Hence
\[
        w_n\in\Gamma_n(F).
\]
Since each $\Gamma_n(F)$ is non-empty, the functions $d_n(z):=\operatorname{dist}(z,\Gamma_n(F))$ are $1$-Lipschitz on $K$. Additionally we have just shown that $d_n(z) \to 0$ for every $z \in K$. For $\eta > 0$ choose a finite $\eta /3$-net $\{z_1,\ldots,z_M\} \subset K$. For each $z_\ell$ choose $N_\ell$ such that $d_n(z_\ell) < \eta /3$ for $n \geq N_\ell$. If $N=\max_\ell N_\ell$ then for every $z \in K$ and $n\geq N$ choosing $\ell$ with $|z-z_\ell| < \eta /3$ gives
\[
d_n(z) \leq |z-z_\ell | + d_n(z_\ell) < \eta,
\]
hence
\[
        \sup_{z\in K} d_n(z) \to 0.
\]
Together with
\[
        \Gamma_n(F)\subseteq K,
\]
this proves
\[
        d_H(\Gamma_n(F),K)\to 0.
\]
Therefore
\[
        \lim_{n\to\infty}\lim_{m\to\infty}\Gamma_{n,m}(F) = R_{\mathrm{ap},\varepsilon}(K_F)
\]
in \(\mathcal M_H\). Since by definition
\[
        \Omega_\X^m\subseteq \Omega_\X,
\]
the same tower gives the same upper bound on \(\Omega_\X^m\).
\end{proof}

\begin{theorem}[Known-modulus upper bound for the regularized target]\label{thm:R-upper-known-modulus}
Let \(1<p<\infty\), let \(\varepsilon>0\), and let \(\alpha\) be a prescribed modulus of continuity. Then
\[
        R_{\mathrm{ap},\varepsilon}(\K_F)\in\Sigma_1^G \quad\text{on }\Omega_X^\alpha.
\]
Consequently,
\[
        R_{\mathrm{ap},\varepsilon}(\K_F)\in\Sigma_1^G \quad\text{on }\Omega_\X^{\alpha,m}.
\]
\end{theorem}

\begin{proof}
Let $\tau_n, \, G_n$ be as in \cref{thm:R-upper-nomodulus}. By \cref{prop:quadrature-convergence}, because \(F\in\Omega_\X^\alpha\), there is a scheduled quadrature level \(m(n)\), depending only on
\[
        n,\, \alpha,\, \text{and the dictionary functions},
\]
such that
\[
        \sup_{z\in G_n}|r_{n,m(n)}(z,F)-r_n(z,F)| < \eta_n,
\]
where
\[
        0<\eta_n<\frac{\tau_n}{4}.
\]
Define the finite base output
\[
        \Gamma_n(F) := \{z\in G_n : r_{n,m(n)}(z,F)<\varepsilon-\tau_n\}.
\]
Again $\Gamma_n(F)$ is non-empty, because $1\in G_n$ and $r_{n,m(n)}(1,F)=0<\varepsilon-\tau_n$. This is a general algorithm, because it uses only the finite point-evaluation table
\[
        \{F(x_{m(n),P}): P \in\mathcal P_{m(n)}\}.
\]

We first prove that
\[
        \Gamma_n(F)\subseteq R_{\mathrm{ap},\varepsilon}(\K_F).
\]
If \(z\in\Gamma_n(F)\), then
\[
        r_{n,m(n)}(z,F) < \varepsilon-\tau_n.
\]
Hence
\[
        r_n(z,F) \le r_{n,m(n)}(z,F)+\eta_n < \varepsilon-\tau_n+\eta_n < \varepsilon.
\]
Since \(\nu_F(z)\le r_n(z,F)\), we get
\[
        \nu_F(z)< \varepsilon,
\]
and therefore \(z\in R_{\mathrm{ap},\varepsilon}(\K_F)\).

Conversely, let
\[
        K:=R_{\mathrm{ap},\varepsilon}(\K_F).
\]
Fix \(z\in K\) and \(\eta>0\). Choose \(w\) with
\[
        |z-w|<\eta/3, \quad \nu_F(w)<\varepsilon.
\]
Choose \(\delta>0\) such that $\delta < \eta /3$ and
\[
        \nu_F(w)+5\delta<\varepsilon.
\]
For all sufficiently large \(n\), we have
\[
        \tau_n+\eta_n<\delta,
\]
there exists \(w_n\in G_n\) with
\[
        |w_n-w|<\delta,
\]
and
\[
        r_n(w_n,F)<\nu_F(w_n)+\delta.
\]
Since \(\nu_F\) is \(1\)-Lipschitz,
\[
        \nu_F(w_n)<\nu_F(w)+\delta.
\]
Thus
\[
        r_n(w_n,F)<\nu_F(w)+2\delta.
\]
The scheduled quadrature error gives
\[
        r_{n,m(n)}(w_n,F) < \nu_F(w)+2\delta+\eta_n.
\]
Since \(\nu_F(w)+5\delta<\varepsilon\) and \(\tau_n+\eta_n<\delta\), this implies
\[
        r_{n,m(n)}(w_n,F)<\varepsilon-\tau_n.
\]
Therefore
\[
        w_n\in\Gamma_n(F).
\]
As in the proof of \cref{thm:R-upper-nomodulus},
\[
        \operatorname{dist}(z,\Gamma_n(F))<\eta
\]
for all sufficiently large \(n\). Hence
\[
        d_H(\Gamma_n(F),K)\to 0.
\]
Thus \(R_{\mathrm{ap},\varepsilon}(\K_F)\) is computed by a one-limit general tower on \(\Omega_\X^\alpha\).
\end{proof}

We recall for the following two corollaries the standard result: If $K_1 \supset K_2 \supset \dots$ are non-empty compact subsets of a compact metric space and $K=\cap_{j\geq 1} K_j$, then $d_H(K_j,K) \to 0$. It will be used in the proofs implicitly: Applying it to the decreasing compact families $R_{\mathrm{ap},\varepsilon + 2^{-k}}(\K_F)$ and $R_{\mathrm{ap},2^{-k}}(\K_F)$ gives the extra outer limit and hence the stated $\Pi^G_k$-bounds.

\begin{corollary}[Closed fixed-\(\varepsilon\) upper bounds]\label{cor:closed-upper-bounds}
Let \(1<p<\infty\) and \(\varepsilon>0\). Then
\[
\begin{array}{c|cccc}
\text{target}
&\Omega_\X
&\Omega_\X^\alpha
&\Omega_\X^m
&\Omega_\X^{\alpha,m}
\\
\hline
C_{\mathrm{ap},\varepsilon}(\K_F)
&\Pi_3^G
&\Pi_2^G
&\Sigma_2^G
&\Sigma_1^G .
\end{array}
\]
\end{corollary}

\begin{proof}
This follows directly from \cref{lem:basic-lower-norm-relations}, \cref{thm:isometry-closed-regularized}, \cref{thm:R-upper-nomodulus} and \cref{thm:R-upper-known-modulus}.
\end{proof}

\begin{corollary}[Exact approximate point spectrum upper bounds]\label{cor:ap-upper-bounds}
Let \(1<p<\infty\). Then
\[
\begin{array}{c|cccc}
\text{target}
&\Omega_\X
&\Omega_\X^\alpha
&\Omega_\X^m
&\Omega_\X^{\alpha,m}
\\
\hline
\sigma_{\mathrm{ap}}(\K_F)
&\Pi_3^G
&\Pi_2^G
&\Pi_3^G
&\Pi_2^G .
\end{array}
\]
\end{corollary}

\begin{proof}
This follows also directly from \cref{lem:basic-lower-norm-relations}, \cref{thm:isometry-closed-regularized}, \cref{thm:R-upper-nomodulus} and \cref{thm:R-upper-known-modulus}.
\end{proof}

We collect the class-wise consequences of the residual construction. The classes were defined in \cref{def:input-classes}; the role of the following section is to state the resulting upper table and then record the lower witnesses proved in this paper.

\section{SCI Upper Bounds And Lower Witnesses}\label{sec:classification-lower-witnesses}

We use the four classes from \cref{def:input-classes}, namely
\[
        \Omega_\X,\,
        \Omega_\X^\alpha,\,
        \Omega_\X^m,\,
        \Omega_\X^{\alpha,m}.
\]
The upper-bound part has already been proved in \cref{sec:residual-towers}. We first collect the consequences in a single table, and then prove the lower witnesses used in this paper.

\begin{remark}[No Hilbert-space reduction]
The proof does \textit{not} use an (Birkhoff-James-) orthogonal Wold-von Neumann decomposition. This is important because the \(L^2\) normal/isometric structure does not give a directly computable analogue on general \(L^p\), \(p\neq2\). All upper bounds are obtained from the residuals of the actual functions
\[
        g\circ F-zg,
\]
where \(g\) ranges over finite-dimensional trial spaces of observables and \(z\in\mathbb C\) is the spectral parameter.
\end{remark}

\begin{theorem}[SCI upper bounds]\label{thm:SCI-upper-bounds-table}
Let \(1<p<\infty\) and \(\varepsilon>0\). Then
\[
\begin{array}{c|cccc}
\text{target}
&\Omega_\X
&\Omega_\X^\alpha
&\Omega_\X^m
&\Omega_\X^{\alpha,m}
\\
\hline
R_{\mathrm{ap},\varepsilon}(\K_F)
&\Sigma_2^G
&\Sigma_1^G
&\Sigma_2^G
&\Sigma_1^G
\\
C_{\mathrm{ap},\varepsilon}(\K_F)
&\Pi_3^G
&\Pi_2^G
&\Sigma_2^G
&\Sigma_1^G
\\
\sigma_{\mathrm{ap}}(\K_F)
&\Pi_3^G
&\Pi_2^G
&\Pi_3^G
&\Pi_2^G .
\end{array}
\]
\end{theorem}

\subsection{A Fixed-\texorpdfstring{\(\varepsilon\)}{epsilon} Measure-Preserving Witness}\label{sec:measure-preserving-witness}

We now prove the fixed-\(\varepsilon\) lower witnesses used in this paper. The first ingredient is the exact lower-norm identity for irrational rotations. This is the only rotation spectral identity used below. In particular, no formula for finite-cycle \(L^p\)-pseudospectra is used.

\begin{proposition}[Spectral distance identity for irrational rotations on $L^{p}$]\label{prop:rot_spectral_identity_Lp}
Fix $1<p<\infty$ and let $\mathcal X=\mathbb T^1=\mathbb R/\mathbb Z$ with Haar probability measure $\nu$. For $\theta\in[0,1)$ define $F_\theta(x)=x+\theta \pmod 1$ and the Koopman isometry $U_\theta:=\mathcal K_{F_\theta}:L^{p}(\mathbb T^1)\to L^{p}(\mathbb T^1)$, $U_\theta f=f\circ F_\theta$.
If $\theta\notin\mathbb Q$, then for every $z\in\mathbb C$
\[
   \sigma_{\inf}(U_\theta-zI)
   \;=\;\operatorname{dist}\!\bigl(z,\mathbb T\bigr)
   \;=\;\bigl||z|-1\bigr|.
\]
In particular, \[
R_{\mathrm{ap},\varepsilon}(U_\theta)
=
C_{\mathrm{ap},\varepsilon}(U_\theta)
=
\{z:\operatorname{dist}(z,\mathbb T)\le\varepsilon\}.
\] for all $\varepsilon>0$.
\end{proposition}

\begin{proof}
We split the proof into three steps.

\smallskip\noindent\textbf{Step 1: Lower bound:}
Since $U_\theta$ is an isometry on $L^p$ (measure preservation of $F_\theta$),
for every $f$ with $\|f\|_p=1$ the reverse triangle inequality gives
\[
   \| (U_\theta-zI)f \|_p
   \;\ge\; \bigl| \|U_\theta f\|_p - \|z f\|_p \bigr|
   \;=\; \bigl| 1 - |z| \bigr|.
\]
Taking the infimum over such $f$ yields $\sigma_{\inf}(U_\theta-zI) \ge ||z|-1| = \operatorname{dist}(z,\mathbb T)$.

\smallskip\noindent\textbf{Step 2: Upper bound:}
When $\theta\notin\mathbb Q$, the character set $\{e_k(x):=e^{2\pi i k x}\}_{k\in\mathbb Z}\subset L^\infty\cap L^p$ consists of eigenfunctions of $U_\theta$ with eigenvalues $\lambda_k=e^{2\pi i k\theta}$ and $\{\lambda_k\}$ is dense in $\mathbb T$ (see e.g., \cite[Prop.~14.24]{eisner2015operator}).
Fix $z\in\mathbb C$ and $\eta>0$. Choose $k$ so that $|\lambda_k - z/|z||<\eta$ (if $z=0$ omit this step).
Since $\|e_k\|_p=1$ (Haar probability), we have
\[
   \|(U_\theta - zI)e_k\|_p
   \;=\; |\lambda_k - z|\,\|e_k\|_p
   \;=\; |\lambda_k - z|
   \;\le\; |\lambda_k - z/|z|| + |z/|z| - z|
   \;\le\; \eta + \bigl||z|-1\bigr|.
\]
Taking the infimum over $\|f\|_p=1$ gives $\sigma_{\inf}(U_\theta - zI) \le \eta + ||z|-1|$; letting $\eta\downarrow0$ yields $\sigma_{\inf}(U_\theta - zI) \le ||z|-1|$. If \(z=0\), then for any character \(e_k\),
\[
        \|(U_\theta-0I)e_k\|_p=1=\bigl||0|-1\bigr|,
\]
so the upper bound is immediate.

\smallskip\noindent\textbf{Step 3:}
Combining steps 1–2 proves $\sigma_{\inf}(U_\theta-zI)=||z|-1|=\operatorname{dist}(z,\mathbb T)$ for all $z$.
In particular, $\sigma_{\inf}(U_\theta-\lambda I)=0$ iff $|\lambda|=1$, so $\mathbb T\subseteq\sigma_{\mathrm{ap}}(U_\theta)$. Conversely, $\sigma_{\mathrm{ap}}(U_\theta)\subseteq\sigma(U_\theta)\subseteq\mathbb T$
(spectrum of an invertible isometry lies on $\mathbb T$), hence $\sigma_{\mathrm{ap}}(U_\theta)=\mathbb T$.
Finally, by the identity just proved,
\[
        C_{\mathrm{ap},\varepsilon}(U_\theta)
        =
        \{z:\sigma_{\inf}(U_\theta-zI)\le\varepsilon\}
        =
        \{z:\operatorname{dist}(z,\mathbb T)\le\varepsilon\}.
\]
Since the function \(z\mapsto\operatorname{dist}(z,\mathbb T)\) has no positive
local minimum at level \(\varepsilon\), the strict sublevel set
\[
        \{z:\operatorname{dist}(z,\mathbb T)<\varepsilon\}
\]
has closure
\[
        \{z:\operatorname{dist}(z,\mathbb T)\le\varepsilon\}.
\]
Hence
\[
        R_{\mathrm{ap},\varepsilon}(U_\theta)
        =
        C_{\mathrm{ap},\varepsilon}(U_\theta)
        =
        \{z:\operatorname{dist}(z,\mathbb T)\le\varepsilon\}
        =
        \mathbb T+\overline B_\varepsilon(0).
\]
\end{proof}

The next construction turns this spectral separation into a finite-query obstruction while keeping a fixed modulus of continuity. A finite algorithm can query only finitely many components of a compact wedge of shrinking circles; an irrational rotation can therefore be hidden on an unqueried component without changing the finite transcript.

\begin{theorem}[Finite-algorithm sharpness on a known-modulus measure-preserving witness]\label{thm:wedge-lower}
Let
\[
        \X_*:=\{*\}\cup\bigsqcup_{j\ge1}\X_j, \qquad \X_j:=\{j\}\times\mathbb T^1.
\]
Put
\[
        r_j:=2^{-j}.
\]
Let \(d_{\mathbb T}\) be the geodesic metric on \(\mathbb T^1\), normalized such that \(d_{\mathbb T}\le 1\). Define a metric \(d_*\) on \(\X_*\) by
\[
        d_*(*,(j,x)):=r_j,
\]
\[
        d_*((j,x),(k,y)):=r_j+r_k
\]
for $j\neq k$, and
\[
        d_*((j,x),(j,y)):=r_j d_{\mathbb T}(x,y).
\]
Let
\[
        \omega_*(\{*\})=0, \quad \omega_*|_{\X_j}=2^{-j}m_{\mathbb T},
\]
where \(m_{\mathbb T}\) is the Haar probability measure on \(\mathbb T^1\). For
\[
        \vartheta=(\vartheta_j)_{j\ge1}\in[0,1)^{\mathbb N},
\]
define
\[
        F_\vartheta(*)=*, \quad F_\vartheta(j,x)=(j,x+\vartheta_j\!\!\!\pmod 1).
\]
Then each \(F_\vartheta\) is measure-preserving and \(1\)-Lipschitz. For every \(\varepsilon>0\), neither
\[
        F\mapsto R_{\mathrm{ap},\varepsilon}(\K_F)
\]
nor
\[
        F\mapsto C_{\mathrm{ap},\varepsilon}(\K_F)
\]
is computable by a finite general algorithm on $\Omega_{\X_*}^{\alpha,m}$, where $\alpha(t)=t$.
\end{theorem}

\begin{proof}
We will split the proof for the sake of readability into four steps.

\smallskip
\noindent\textbf{Step 1: compactness and the modulus:}
The metric \(d_*\) is the standard wedge metric (therefore we will just sketch how to show that this is indeed a metric), see, e.g. \cite[Sec.~2.6-2.8]{adamaszek2020homotopy}: each component \(\X_j\) has diameter at most \(r_j\), the distance from \(\X_j\) to the base point is \(r_j\), and the distance between distinct components is the sum of their distances to the base point. The triangle inequality follows by checking: inside a single component it is the triangle inequality for \(d_{\mathbb T}\); between distinct components it is equality through the base point; and if two points lie in one component while the third lies outside, then
\[
        r_j d_{\mathbb T}(x,y)\le r_j\le r_j+2r_k = d_*((j,x),(k,z))+d_*((k,z),(j,y)).
\]
Since \(r_j\to 0\), every sequence either has a subsequence in one compact component \(\X_j\), or has components \(j\to\infty\) and then converges to \(*\). Thus \(\X_*\) is compact.

For each \(\vartheta\), the map \(F_\vartheta\) fixes \(*\), preserves every component \(\X_j\), and acts on \(\X_j\) by a circle rotation. Therefore it is an isometry on every component. Distances between different components and distances to \(*\) depend only on the component indices, which are preserved. Hence
\[
        d_*(F_\vartheta a,F_\vartheta b)=d_*(a,b)
\]
for $a,b\in \X_*$, so \(F_\vartheta\) is \(1\)-Lipschitz. Since rotations preserve the Haar measure on each \(\X_j\), and since \(F_\vartheta\) fixes the null point \(*\), we get
\[
        (F_\vartheta)_\#\omega_*=\omega_*.
\]
Thus
\[
        F_\vartheta\in\Omega_{\X_*}^{\alpha,m},
\]
with $\alpha(t)=t$.

\smallskip
\noindent\textbf{Step 2: indistinguishability by a finite algorithm:}
Run $\Gamma$ on the identity map
\[
   F_0:=F_{(0,0,\ldots)}=\operatorname{id}_{\X_*}.
\]
The finite transcript of $\Gamma$ on $F_0$ consists of finitely many oracle queries
\[
   \lambda_{x,\varphi}(F_0)=\varphi(F_0(x)).
\]
Let $Q_\X \subset \X_*$ be the finite set of domain points $x$ which occur in these queries. Since there are infinitely many components $(\X_j)_{j\ge1}$, there exists a $j_0$ such that
\[
   Q_\X \cap \X_{j_0}=\varnothing.
\]

Choose an irrational number
\[
        \beta\in[0,1)\setminus\mathbb Q.
\]
Define \(F_1\) by
\[
        F_1=F_0 \quad\text{on }\X_*\setminus \X_{j_0},
\]
and
\[
        F_1(j_0,x)=(j_0,x+\beta\!\!\!\pmod 1).
\]
Then \(F_1\in\Omega_{\X_*}^{\alpha,m}\). For every queried pair $(x,\varphi)$ in the transcript on $F_0$, we have $F_1(x)=F_0(x)$, and hence
\[
   \lambda_{x,\varphi}(F_1)=\lambda_{x,\varphi}(F_0).
\]
By the consistency axiom we get
\[
   \Gamma(F_1)=\Gamma(F_0).
\] 

\smallskip
\noindent\textbf{Step 3: the two target sets:}
For \(F_0=\operatorname{id}_{\X_*}\),
\[
        \K_{F_0}=I,
\]
and hence
\[
        \sigma_{\inf}(\K_{F_0}-zI)=|1-z|.
\]
Therefore
\[
        R_{\mathrm{ap},\varepsilon}(\K_{F_0}) = C_{\mathrm{ap},\varepsilon}(\K_{F_0}) = \overline B(1,\varepsilon).
\]

For \(F_1\), the space decomposes as an \(L^p\)-direct sum
\[
        L^p(\X_*,\omega_*) = \bigoplus_{j\ge1}^{p} L^p(\X_j,2^{-j}m_{\mathbb T})
\]
up to the null point \(*\). The Koopman operator \(\K_{F_1}\) is the direct sum of identity operators on all components \(\X_j\), \(j\neq j_0\), and the irrational rotation operator \(U_\beta\) on \(\X_{j_0}\).

For an \(L^p\)-direct sum \(T=\bigoplus_j T_j\), the lower norm satisfies
\[
        \sigma_{\inf}(T-zI) = \inf_j \sigma_{\inf}(T_j-zI).
\]
Indeed, for every \(f=(f_j)\) with \(\|f\|_p= 1\),
\[
        \|(T-zI)f\|_p^p = \sum_j \|(T_j-zI)f_j\|_p^p \ge \left(\inf_j\sigma_{\inf}(T_j-zI)\right)^p \sum_j\|f_j\|_p^p = \left(\inf_j\sigma_{\inf}(T_j-zI)\right)^p.
\]
The reverse inequality is obtained by taking \(f\) supported in a single block and then taking the infimum over the block index. Consequently,
\[
        \sigma_{\inf}(\K_{F_1}-zI) = \min\left\{ |1-z|, \sigma_{\inf}(U_\beta-zI) \right\} = \min\left\{ |1-z|, \operatorname{dist}(z,\mathbb T) \right\}.
\]
Since \(1\in\mathbb T\),
\[
        \operatorname{dist}(z,\mathbb T)\le |z-1|.
\]
Therefore
\[
        \sigma_{\inf}(\K_{F_1}-zI) = \operatorname{dist}(z,\mathbb T).
\]
By \cref{prop:rot_spectral_identity_Lp},
\[
        R_{\mathrm{ap},\varepsilon}(\K_{F_1}) = C_{\mathrm{ap},\varepsilon}(\K_{F_1}) = \{z:\operatorname{dist}(z,\mathbb T)\le\varepsilon\}.
\]

\smallskip
\noindent\textbf{Step 4: Hausdorff separation:}
Denote for short
\[
        A_\varepsilon:=\overline B(1,\varepsilon), \quad B_\varepsilon:=\{z : \operatorname{dist}(z,\mathbb T)\le\varepsilon\}.
\]
Since
\[
        |z-1|\le\varepsilon \, \Longrightarrow \, \operatorname{dist}(z,\mathbb T)\le\varepsilon,
\]
we have
\[
        A_\varepsilon\subseteq B_\varepsilon.
\]
The point
\[
        z_\varepsilon:=-(1+\varepsilon)
\]
belongs to \(B_\varepsilon\), because
\[
        \operatorname{dist}(-(1+\varepsilon),\mathbb T)=\varepsilon.
\]
Its distance to \(A_\varepsilon\) is
\[
        \operatorname{dist}(z_\varepsilon,A_\varepsilon) = |z_\varepsilon-1|-\varepsilon = (2+\varepsilon)-\varepsilon = 2.
\]
Conversely, for every \(z\in B_\varepsilon\),
\[
        |z|\le 1+\varepsilon,
\]
and therefore
\[
        \operatorname{dist}(z,A_\varepsilon) \le \max\{|z-1|-\varepsilon,0\} \le |z|+1-\varepsilon \le 2.
\]
Thus
\[
        d_H(A_\varepsilon,B_\varepsilon)=2.
\]

Hence the two target sets for \(F_0\) and \(F_1\) are different, while \(\Gamma(F_0)=\Gamma(F_1)\), contradiction. Therefore no finite general algorithm computes \(R_{\mathrm{ap},\varepsilon}\) or \(C_{\mathrm{ap},\varepsilon}\) on \(\Omega_{\X_*}^{\alpha,m}\).
\end{proof}

Combining the finite-algorithm obstruction with the one-limit upper bound on \(\Omega_{\X_*}^{\alpha,m}\) gives sharpness on this witness class.

\begin{corollary}[Sharpness on the known-modulus measure-preserving witness]\label{cor:wedge-sharpness}
On the witness space \(\X_*\) of \cref{thm:wedge-lower} we have
\[
        R_{\mathrm{ap},\varepsilon}(\K_F), \, C_{\mathrm{ap},\varepsilon}(\K_F) \in \Sigma_1^G\setminus\Delta_0^G \quad \text{on }\Omega_{\X_*}^{\alpha,m}
\]
with $\alpha(t)=t$.
\end{corollary}

\subsection{A Boundary Lower Bound For The Closed Nonsingular Problem}\label{sec:closed-boundary-lower}

The closed target \(C_{\mathrm{ap},\varepsilon}\) can be harder than the regularized target in the nonsingular non-measure-preserving case. The reason is that the boundary condition
\[
        \sigma_{\inf}(\K_F-zI)=\varepsilon
\]
may carry limiting information. The following atomic construction encodes a one-sided infinite-search predicate into the condition
\[
        0\in C_{\mathrm{ap},\varepsilon}(\K_F).
\]
It therefore shows that the additional outer limit in the known-modulus nonsingular closed problem is not merely an artifact of the proof.

\begin{lemma}[No one-sided finite-query tower for the monotone search predicate]\label{lem:monotone-predicate-no-Pi1}
Let
\[
        \mathcal A_{\mathrm{mon}} := \left\{ A=(a_{D,j})_{D,j\ge1}\in\{0,1\}^{\mathbb N^2} : a_{D+1,j}\le a_{D,j}\ \forall D,j \right\}.
\]
and
\[
        P(A)=1 \quad\Longleftrightarrow\quad \forall D\ge1\ \exists j\ge1\ a_{D,j}=1.
\]
There is no sequence of finite-coordinate general algorithms
\[
        \Delta_n:\mathcal A_{\mathrm{mon}}\to\{0,1\},
        \qquad n\in\mathbb N,
\]
such that
\[
        P(A)=1
        \quad\Longrightarrow\quad
        \Delta_n(A)=1\ \forall n,
\]
and
\[
        P(A)=0
        \quad\Longrightarrow\quad
        \Delta_n(A)=0
        \text{ for all sufficiently large }n.
\]
\end{lemma}

\begin{proof}
Assume that such a sequence \(\Delta_n\) exists. We first prove the following claim: if a finite-coordinate general algorithm \(\Delta : \mathcal A_{\mathrm{mon}}\to\{0,1\}\) satisfies
\[
        P(B)=1 \quad\Longrightarrow\quad \Delta(B)=1,
\]
then
\[
        \Delta(A)=1 \qquad \forall A\in\mathcal A_{\mathrm{mon}}.
\]

Indeed, suppose for contradiction that \(\Delta(A)=0\) for some \(A\in\mathcal A_{\mathrm{mon}}\). Let \(Q\subset\mathbb N^2\) be the finite set of coordinates queried by \(\Delta\) on input \(A\). By the consistency axiom for general algorithms, any \(B\in\mathcal A_{\mathrm{mon}}\) satisfying
\[
        b_{D,j}=a_{D,j},
\]
where $(D,j)\in Q$, also satisfies
\[
        \Delta(B)=\Delta(A)=0.
\]
We now construct such a \(B\) with \(P(B)=1\). Let
\[
        J_Q:=\{j : (D,j)\in Q\text{ for some }D\}.
\]
Choose $j_*\notin J_Q$. For each \(j\in J_Q\), define
\[
        L_j:=\max\{D : (D,j)\in Q\text{ and }a_{D,j}=1\},
\]
with the convention \(L_j=0\) if the set is empty. Define \(B\) by
\[
        b_{D,j_*}:=1
\]
for $D\ge1$, and
\[
        b_{D,j}:=
        \begin{cases}
        1,&D\le L_j,\\
        0,&D>L_j,
        \end{cases}
\]
where $j\in J_Q$ and $b_{D,j}:=0$ for $j\notin J_Q\cup\{j_*\}$. Then \(B\in\mathcal A_{\mathrm{mon}}\), because in each column the sequence \(D\mapsto b_{D,j}\) is nonincreasing. Moreover \(B\) agrees with \(A\) on every queried coordinate. To check this, let \((D,j)\in Q\). If \(a_{D,j}=1\), then \(D\le L_j\), so \(b_{D,j}=1\). If \(a_{D,j}=0\), then no queried row \(E\ge D\) in the same column can satisfy \(a_{E,j}=1\), because \(A\in\mathcal A_{\mathrm{mon}}\) would imply
\[
        a_{D,j}\ge a_{E,j}=1,
\]
contradiction. Hence \(L_j<D\), so \(b_{D,j}=0\).

Finally $b_{D,j_*}=1$ for $D\ge1$, so $P(B)=1$. This contradicts the assumption that \(\Delta(B)=1\) for every \(B\) with \(P(B)=1\). The claim is proved.

Apply the claim to each \(\Delta_n\). Since \(P(A)=1\) implies \(\Delta_n(A)=1\) for all \(n\), each \(\Delta_n\) satisfies the hypothesis of the claim. Therefore
\[
        \Delta_n(A)=1 \qquad \forall A\in\mathcal A_{\mathrm{mon}},\ \forall n.
\]
But there exist \(A\in\mathcal A_{\mathrm{mon}}\) with \(P(A)=0\), for example \(a_{D,j}=0\) for all \(D,j\). For such \(A\), the assumed eventual condition would require \(\Delta_n(A)=0\) for all sufficiently large \(n\), contradiction.
\end{proof}

\begin{theorem}[Boundary obstruction for the closed nonsingular problem]\label{thm:closed-boundary-lower}
Let \(1<p<\infty\) and \(0<\varepsilon<1\). There exists a compact atomic metric probability space
\[
        (\X_\sharp,d_\sharp,\mu_\sharp)
\]
and a fixed modulus
\[
        \alpha(t)=t
\]
such that
\[
        C_{\mathrm{ap},\varepsilon}(\K_F) \in \Pi_2^G(\Omega_{\X_\sharp}^{\alpha}) \setminus \Pi_1^G(\Omega_{\X_\sharp}^{\alpha}).
\]
\end{theorem}

\begin{proof}
The upper bound follows directly from \cref{cor:closed-upper-bounds}. It remains to prove the non-inclusion in \(\Pi_1^G\). Since the proof is a bit long we will split it for the sake of readability into five steps.

\smallskip
\noindent\textbf{Step 1: the compact atomic space:}
Set
\[
        \lambda_D:=\varepsilon+(1-\varepsilon)2^{-D},
\]
where $D\ge1$. Then
\[
        \varepsilon<\lambda_D<1, \quad \lambda_D\downarrow\varepsilon.
\]
Let
\[
        \X_\sharp := \{*\}\cup\{u_{D,j},v_{D,j} : D,j\ge1\}.
\]
Put
\[
        r_{D,j}:=2^{-(D+j+2)}.
\]
Define \(d_\sharp\) by
\[
        d_\sharp(*,u_{D,j})=d_\sharp(*,v_{D,j})=r_{D,j},
\]
\[
        d_\sharp(u_{D,j},v_{D,j})=\frac12 r_{D,j},
\]
and, for \((D,j)\neq(E,k)\),
\[
        d_\sharp(x_{D,j},y_{E,k})=r_{D,j}+r_{E,k},
\]
where
\[
        x_{D,j}\in\{u_{D,j},v_{D,j}\}, \quad y_{E,k}\in\{u_{E,k},v_{E,k}\}.
\]
This map is a metric. To see this, we note that the only nontrivial triangle cases are those involving two points in the same two-point component and a third point outside it; for those,
\[
        d_\sharp(u_{D,j},v_{D,j}) = \frac12 r_{D,j} \le 2 r_{D,j}+2r_{E,k} = d_\sharp(u_{D,j},y_{E,k})+d_\sharp(y_{E,k},v_{D,j}).
\]
All other cases either reduce to equality through the base point \(*\), or are immediate.

Since \(r_{D,j}\to 0\) as \(D+j\to\infty\), every sequence in \(\X_\sharp\) has a subsequence either eventually contained in a finite subset of \(\X_\sharp\), or converging to \(*\). Hence \(\X_\sharp\) is  compact. Define now
\[
        b_{D,j}:=2^{-(D+j+2)}.
\]
Since
\[
        \sum_{D,j\ge1}b_{D,j} = \frac14,
\]
and \(0<\lambda_D^p<1\), we get
\[
        \sum_{D,j\ge1}(1+\lambda_D^p)b_{D,j} < 2\sum_{D,j\ge1}b_{D,j} = \frac12.
\]
Define the probability measure \(\mu_\sharp\) by
\[
        \mu_\sharp(v_{D,j})=b_{D,j}, \quad \mu_\sharp(u_{D,j})=\lambda_D^p b_{D,j},
\]
and
\[
        \mu_\sharp(\{*\}) = 1-\sum_{D,j\ge1}(1+\lambda_D^p)b_{D,j}.
\]
Clearly, every point has positive measure by construction, so this defines indeed a probability measure.

\smallskip\noindent\textbf{Step 2: the coded maps:}
Let
\[
        A=(a_{D,j})_{D,j\ge1}\in\mathcal A_{\mathrm{mon}},
\]
where $\mathcal A_{\mathrm{mon}}$ is defined as in \cref{lem:monotone-predicate-no-Pi1}. Define \(F_A:\X_\sharp\to \X_\sharp\) by $F_A(*)=*$, and, on each two-point component,
\[
        F_A(u_{D,j}) =
        \begin{cases}
        v_{D,j},&a_{D,j}=1,\\
        u_{D,j},&a_{D,j}=0,
        \end{cases}
\]
\[
        F_A(v_{D,j}) =
        \begin{cases}
        u_{D,j},&a_{D,j}=1,\\
        v_{D,j},&a_{D,j}=0.
        \end{cases}
\]
Thus \(F_A\) is either the transposition of \(u_{D,j}\) and \(v_{D,j}\), or the identity, depending on \(a_{D,j}\).

Every \(F_A\) is an isometry of the metric space \(\X_\sharp\). Indeed, on each two-point component it either fixes both points or swaps them, and hence preserves the ''within-component'' distance. It fixes the base point \(*\), preserves the component labels \((D,j)\), and therefore also preserves all distances between different components and all distances to \(*\). Hence
\[
        F_A\in\Omega_{\X_\sharp}^{\alpha},
\]
with $\alpha(t)=t$.

\smallskip\noindent\textbf{Step 3: the Radon-Nikodym density:}
The map \(F_A\) is a bijection of the countable set \(\X_\sharp\), and every atom has positive measure. Hence \(F_A\) is nonsingular. The Radon-Nikodym density
\[
        \rho_A:=\frac{d(F_A)_\#\mu_\sharp}{d\mu_\sharp}
\]
is given pointwise by
\[
        \rho_A(x)=\frac{\mu_\sharp(F_A^{-1}\{x\})}{\mu_\sharp(\{x\})}.
\]
If \(a_{D,j}=0\), then \(F_A\) is the identity on the pair \(\{u_{D,j},v_{D,j}\}\), and therefore
\[
        \rho_A(u_{D,j})=\rho_A(v_{D,j})=1.
\]
If \(a_{D,j}=1\), then \(F_A\) swaps \(u_{D,j}\) and \(v_{D,j}\), so
\[
        \rho_A(v_{D,j}) = \frac{\mu_\sharp(u_{D,j})}{\mu_\sharp(v_{D,j})} = \lambda_D^p,
\]
and
\[
        \rho_A(u_{D,j}) = \frac{\mu_\sharp(v_{D,j})}{\mu_\sharp(u_{D,j})} = \lambda_D^{-p}.
\]
Also $\rho_A(*)=1$. Since \(\lambda_D>\varepsilon\),
\[
        \rho_A(x)\le\varepsilon^{-p}
\]
for $x\in \X_\sharp$. Thus
\[
        \rho_A\in L^\infty(\X_\sharp,\mu_\sharp),
\]
and \(\K_{F_A}\) is bounded.

\smallskip\noindent\textbf{Step 4: exact lower norm at the origin:}
By \cref{prop:boundednessLp}, for every \(f\in L^p(\X_\sharp,\mu_\sharp)\),
\[
        \|\K_{F_A}f\|_p^p = \int_{\X_\sharp}|f|^p\rho_A\,d\mu_\sharp.
\]
Therefore
\[
        \sigma_{\inf}(\K_{F_A}) = \operatorname*{ess\,inf}_{x\in \X_\sharp}\rho_A(x)^{1/p}.
\]
Indeed, the inequality
\[
        \|\K_{F_A}f\|_p \ge \left(\operatorname*{ess\,inf}_{x}\rho_A(x)^{1/p}\right)\|f\|_p
\]
holds for every \(f\). Conversely, because the space is purely atomic and all atoms have positive measure, the lower bound is attained or approximated by normalized indicators of atoms on which \(\rho_A^{1/p}\) attains (or approaches) its essential infimum. Using the density values from Step~3, we obtain
\[
        \sigma_{\inf}(\K_{F_A}) = \min\left\{1,\, \inf\{\lambda_D:a_{D,j}=1\text{ for some }j\} \right\},
\]
with the convention that the infimum over the empty set is \(+\infty\). It follows that
\[
        0\in C_{\mathrm{ap},\varepsilon}(\K_{F_A}) \quad\Longleftrightarrow\quad \sigma_{\inf}(\K_{F_A})\le\varepsilon \quad\Longleftrightarrow\quad P(A)=1.
\]
Indeed, if \(P(A)=1\), then for every \(D\) there exists \(j\) such that \(a_{D,j}=1\). Hence
\[
        \inf\{\lambda_D : a_{D,j}=1\text{ for some }j\} = \inf_D \lambda_D = \varepsilon.
\]
Conversely, if \(P(A)=0\), then there exists \(D_0\) such that
\[
        a_{D_0,j}=0 \qquad \forall j.
\]
Since \(A\in\mathcal A_{\mathrm{mon}}\), this implies
\[
        a_{D,j}=0 \qquad \forall D\ge D_0,\ \forall j.
\]
Thus active entries \(a_{D,j}=1\) can occur only in the finitely many rows \(D<D_0\). Hence
\[
        \sigma_{\inf}(\K_{F_A}) \ge \min\{1,\lambda_1,\ldots,\lambda_{D_0-1}\} > \varepsilon.
\]

\smallskip\noindent\textbf{Step 5: contradiction to \(\Pi_1^G\):}
Assume, for contradiction, that
\[
        C_{\mathrm{ap},\varepsilon}(\K_F)\in \Pi_1^G(\Omega_{\X_\sharp}^{\alpha}).
\]
By the definition of \(\Pi_1^G\), there is a one-limit tower of general algorithms
\[
        \Gamma_n:\Omega_{\X_\sharp}^{\alpha}\to \mathcal M_H
\]
converging to $C_{\mathrm{ap},\varepsilon}(\K_F)$, together with outer sets \(X_n(F)\in \mathcal M_H\) satisfying
\[
        C_{\mathrm{ap},\varepsilon}(\K_F)\subseteq X_n(F),
\]
and
\[
        d_H(X_n(F),\Gamma_n(F))\le 2^{-n}.
\]
Replacing \(X_n(F)\) by
\[
        Y_n(F):= \{z\in\mathbb C:\operatorname{dist}(z,\Gamma_n(F))\le 2^{-n}\},
\]
we may assume that the outer approximants are obtained from the general algorithm output itself. Indeed,
\[
        C_{\mathrm{ap},\varepsilon}(\K_F)\subseteq Y_n(F),
\]
because \(C_{\mathrm{ap},\varepsilon}(\K_F)\subseteq X_n(F)\) and \(d_H(X_n(F),\Gamma_n(F))\le2^{-n}\). Moreover,
\[
        Y_n(F)\to C_{\mathrm{ap},\varepsilon}(\K_F)
\]
in Hausdorff distance, since \(\Gamma_n(F)\to C_{\mathrm{ap},\varepsilon}(\K_F)\) and \(2^{-n}\to 0\). The map \(F\mapsto Y_n(F)\) is still a general algorithm, as it is obtained by post-processing \(\Gamma_n(F)\).

For \(A\in\mathcal A_{\mathrm{mon}}\), define now
\[
        \Delta_n(A) :=
        \begin{cases}
        1,&0\in Y_n(F_A),\\
        0,&0\notin Y_n(F_A).
        \end{cases}
\]
Each \(\Delta_n\) is a finite-coordinate general algorithm on \(\mathcal A_{\mathrm{mon}}\). Indeed, a finite transcript consists of finitely many oracle queries
\[
   \lambda_{x,\varphi}(F_A)=\varphi(F_A(x)).
\]
Only finitely many domain points $x\in \X_\sharp$ occur in such a transcript. If $x=\ast$, the value $F_A(x)=\ast$ is independent of $A$. If $x=u_{D,j}$ or $x=v_{D,j}$, then $F_A(x)$ is determined by the single
bit $a_{D,j}$. Hence $\Delta_n$ depends on only finitely many coordinates of $A$.

If \(P(A)=1\), then
\[
        0\in C_{\mathrm{ap},\varepsilon}(\K_{F_A}) \subseteq Y_n(F_A),
\]
i.e. $\Delta_n(A)=1$. If \(P(A)=0\), then Step~4 gives
\[
        \sigma_{\inf}(\K_{F_A})>\varepsilon.
\]
Set
\[
        \delta_A:=\sigma_{\inf}(\K_{F_A})-\varepsilon>0.
\]
Since
\[
        |\sigma_{\inf}(\K_{F_A}-zI)-\sigma_{\inf}(\K_{F_A})| \le |z|,
\]
we have
\[
        B(0,\delta_A/2)\cap C_{\mathrm{ap},\varepsilon}(\K_{F_A})=\emptyset.
\]
Because
\[
        Y_n(F_A)\to C_{\mathrm{ap},\varepsilon}(\K_{F_A})
\]
in Hausdorff distance, it follows that $0\notin Y_n(F_A)$ for all sufficiently large \(n\). Hence $\Delta_n(A)=0$ for all sufficiently large \(n\).

Thus \(\Delta_n\) satisfies the two conditions ruled out by \cref{lem:monotone-predicate-no-Pi1}. This contradiction proves
\[
        C_{\mathrm{ap},\varepsilon}(\K_F) \notin \Pi_1^G(\Omega_{\X_\sharp}^{\alpha}).
\]
\end{proof}

We summarize the results proved in this work.

\begin{table}[h]
\centering
\[
\begin{array}{c|cccc}
\text{target}
&\Omega_\X
&\Omega_\X^\alpha
&\Omega_\X^m
&\Omega_\X^{\alpha,m}
\\
\hline
R_{\mathrm{ap},\varepsilon}
&\Sigma_2^G
&\Sigma_1^G
&\Sigma_2^G
&\Sigma_1^G
\\
C_{\mathrm{ap},\varepsilon}
&\Pi_3^G
&\Pi_2^G
&\Sigma_2^G
&\Sigma_1^G
\\
\sigma_{\mathrm{ap}}
&\Pi_3^G
&\Pi_2^G
&\Pi_3^G
&\Pi_2^G
\end{array}
\]
\caption{Upper bounds proved in \cref{thm:SCI-upper-bounds-table}.}
\label{tab:upper-bounds}
\end{table}

\begin{table}[h]
\centering
\[
\begin{array}{c|c|c}
\text{target}
&\text{witness class}
&\text{proved lower bound}
\\
\hline
R_{\mathrm{ap},\varepsilon}
&\Omega_{\X_*}^{\alpha,m}
&\notin\Delta_0^G
\\
C_{\mathrm{ap},\varepsilon}
&\Omega_{\X_*}^{\alpha,m}
&\notin\Delta_0^G
\\
C_{\mathrm{ap},\varepsilon}
&\Omega_{\X_\sharp}^{\alpha}
&\notin\Pi_1^G
\end{array}
\]
\caption{Lower witness results proved in \cref{thm:wedge-lower} and \cref{thm:closed-boundary-lower}. Here, in the last row, $0<\varepsilon<1$.}
\label{tab:lower-witnesses}
\end{table}

\subsection{The Remaining No-Modulus Exact-Spectrum Lower Bound}\label{sec:openProblem}

We have shown the upper bound
\begin{equation}\label{eq:openProblem}
        \sigma_{\mathrm{ap}}(\K_F)\in\Pi_3^G \quad \text{on }\Omega_\X^m.
\end{equation}
It turns out that the lower bound is highly nontrivial to prove. A natural candidate would be locking all finite-stage outputs to a reference input based on an idea in \cite[Theorem~B.1, Step~5]{colbrook2024limits}. The following elementary observation shows why this is impossible and why \eqref{eq:openProblem} is a serious remaining open problem.

\begin{proposition}[Locked finite stages cannot encode a lower-bound source]\label{prop:locking-no-go}
Let \(\Xi:\Omega\to \mathcal M_H\) be a compact-set-valued computational problem, and
let
\[
        \{\Gamma_{n_2,n_1}\}_{n_2,n_1\in\mathbb N}
\]
be a height-two tower of general algorithms converging to \(\Xi\). Let \(F_0\in\Omega\), and let
\[
        \{F_A:A\in\mathcal A\}\subset\Omega
\]
be a family such that
\[
        \Gamma_{n_2,n_1}(F_A)=\Gamma_{n_2,n_1}(F_0) \qquad \forall A\in\mathcal A,\ \forall n_1,n_2\in\mathbb N.
\]
Then
\[
        \Xi(F_A)=\Xi(F_0) \qquad \forall A\in\mathcal A.
\]
In particular, such a locking construction cannot prove a lower bound by encoding an \(A\)-dependent decision problem into \(\Xi(F_A)\).
\end{proposition}

\begin{proof}
Fix \(A\in\mathcal A\). For every fixed \(n_2\), the inner-limit sequences are identical, i.e.
\[
        \Gamma_{n_2,n_1}(F_A)=\Gamma_{n_2,n_1}(F_0)
\]
for $n_1\in\mathbb N$. Therefore their inner limits are identical:
\[
        \Gamma_{n_2}(F_A) := \lim_{n_1\to\infty}\Gamma_{n_2,n_1}(F_A) = \lim_{n_1\to\infty}\Gamma_{n_2,n_1}(F_0) =: \Gamma_{n_2}(F_0).
\]
Taking the outer limit gives
\[
        \Xi(F_A) = \lim_{n_2\to\infty}\Gamma_{n_2}(F_A) = \lim_{n_2\to\infty}\Gamma_{n_2}(F_0) = \Xi(F_0).
\]
Thus the target value cannot depend on \(A\).
\end{proof}

\begin{problem}[No-modulus exact-spectrum sharpness]\label{prob:p4}
Does one have
\[
        \sigma_{\mathrm{ap}}(\K_F)\notin \Pi_2^G \quad \text{on }\Omega_\X^m?
\]
\end{problem}

\begin{remark}[What a proof of \cref{prob:p4} would roughly require]
A proof cannot be obtained by locking all finite-stage outputs to a reference input, by \cref{prop:locking-no-go}. A viable proof would need a finite-information spectral realization of a height-three source: one would need a compact metric probability space \((\X,\omega)\), a family
\[
        A\mapsto F_A\in\Omega_\X^m,
\]
and a \(\Pi^0_3\)-complete predicate
\[
        \widehat\Xi_3(A) = 1 \iff (\forall n_1)(\exists n_2)(\forall n_3) \bigl[a_{\beta(n_1,n_2,n_3)}=1\bigr],
\]
such that finite point-evaluation transcripts of \(F_A\) are determined by finitely many bits of \(A\), while the compact sets $\sigma_{\mathrm{ap}}(\K_{F_A})$ fall into Hausdorff-separated marker classes according to \(\widehat\Xi_3(A)\). Constructing such a continuous measure-preserving finite-information realization is the missing ingredient.
\end{remark}

\section{Conclusion And Outlook}\label{sec:concl_outlook}

We studied residual SCI bounds for approximate point spectral sets of bounded Koopman operators on \(L^p(\mathcal X,\omega)\), \(1<p<\infty\), under point-evaluation oracles for the underlying map \(F\). A central feature of the Banach-space setting is the distinction between the compact regularized target $R_{\mathrm{ap},\varepsilon}(T)$ and the closed lower-norm target $C_{\mathrm{ap},\varepsilon}(T)$ for $T \in \mathcal B(\Y)$, where $\Y$ is a complex Banach space. The second set is the closed approximate point \(\varepsilon\)-pseudospectrum, whereas the first is the compact object naturally produced by strict residual sublevel tests.

On the upper-bound side, we constructed residual towers using continuous finite-dimensional trial spaces and tagged quadrature residuals for the actual functions
\[
        g\circ F-zg,
\]
where \(g\) ranges over finite-dimensional trial observables and \(z\in\mathbb C\) is the spectral parameter. This avoids the conditional-expectation sampling issue that would arise if one tried to compute martingale truncation norms from point samples of \(F\). The resulting upper bounds can be found in \cref{tab:upper-bounds}.

On the lower-bound side, the fixed-\(\varepsilon\) results proved here are witness results. On a compact wedge of shrinking circles with a fixed Lipschitz modulus, no finite general algorithm computes either
\(R_{\mathrm{ap},\varepsilon}(\K_F)\) or \(C_{\mathrm{ap},\varepsilon}(\K_F)\) on the measure-preserving known-modulus class. Thus the one-limit upper bound is sharp on that witness class. For the closed nonsingular problem, we constructed a compact atomic system in which the boundary condition
\[
        0\in C_{\mathrm{ap},\varepsilon}(\K_F)
\]
encodes a monotone infinite-search predicate.

The no-modulus sharpness of the exact approximate point spectrum remains open. The present work proves
\[
        \sigma_{\mathrm{ap}}(\K_F)\in\Pi_3^G \quad \text{on }\Omega_{\mathcal X}^{m},
\]
but does not claim the matching lower bound
\[
        \sigma_{\mathrm{ap}}(\K_F)\notin\Pi_2^G.
\]
\Cref{prop:locking-no-go} explains why the natural locking strategy cannot prove such a result: if all finite-stage outputs are locked to a reference input, then the iterated limits are locked as well. A proof of \cref{prob:p4} would require a new finite-information continuous measure-preserving realization of a \(\Pi^0_3\)-complete source predicate.

We conclude with several example directions for future research.
\begin{itemize}
\item \textbf{Endpoint spaces \(L^1\) and \(L^\infty\):}
In \(L^1\), one may still hope to build residual towers from dense simple or continuous trial spaces, but the isometry theorem
\[
        C_{\mathrm{ap},\varepsilon}=R_{\mathrm{ap},\varepsilon}
\]
used here for \(1<p<\infty\) depends on complex uniform convexity of \(L^p\), which is no longer available at \(p=1\). Thus the closed target may need a separate consideration. In \(L^\infty\), nonseparability prevents a direct countable dense-trial-space argument in the norm topology; possible replacements include partition nets, weak-\(*\) compactness, or oracle models based on averages over atoms rather than pointwise values. The fixed-\(\varepsilon\) witness mechanisms based on rotations and finite point-evaluation obstructions are expected to remain informative, but endpoint upper bounds require discretizations adapted to the endpoint topology.

\item \textbf{Koopman operators on \(\mathrm{BV}\) and anisotropic spaces:}
Adapting the SCI framework to spaces such as \(\mathrm{BV}\) raises different questions. The space \(\mathrm{BV}\) is not reflexive, step functions are not dense in the \(\mathrm{BV}\)-norm in the same way as in \(L^1\), and the Koopman operator may fail to preserve \(\mathrm{BV}\) unless the underlying dynamics has additional regularity. For classes of maps where the Koopman operator is well-defined on \(\mathrm{BV}\), a plausible route is to use Ulam-type partition schemes and obtain residual control from Lasota-Yorke inequalities; see, for example, \cite{LasotaYorkeInvariant,KellerBV,BaladiPos}. Under suitable $\mathrm{BV}$-preserving hypotheses, one may expect analogues of the known-modulus upper bounds,
\[
        R_{\mathrm{ap},\varepsilon}\in\Sigma_1^G, \qquad \sigma_{\mathrm{ap}}\in\Pi_2^G,
\]
while the closed target \(C_{\mathrm{ap},\varepsilon}\) may again require a separate boundary consideration.

\item \textbf{No-modulus exact-spectrum sharpness:}
The most important missing sharpness question might be \cref{prob:p4}. The known-modulus and fixed-\(\varepsilon\) witnesses in this paper do not resolve it. A successful proof would need to encode a height-three source into the exact approximate point spectrum of continuous measure-preserving Koopman maps, while keeping every finite point-evaluation transcript finitely determined by the source data. This appears to require a fundamental new finite-information measure-preserving spectral realization, rather than a locking argument.
\end{itemize}

\textbf{Acknowledgments } The author thanks Christian Lange for helpful feedback on small typos in the first version of this manuscript, Giovanni Fantuzzi for introducing Koopman operators, Lorenzo Liverani and Gustav Conradie for discussions and Enrique Zuazua for the opportunity of working with the members of the chair FAU-DCN. The author is also grateful to the anonymous referee for a careful and detailed report on an earlier version of this work. Several corrections and improvements in the present version could be found in (or were inspired by) the report, including the distinction between the regularized and closed approximate point \(\varepsilon\)-pseudospectra, the replacement of the earlier conditional-expectation sampling argument by tagged quadrature residuals, and the reformulation of the no-modulus exact-spectrum lower bound as an open problem.

\textbf{Funding } This research did not receive any specific grant from funding agencies in the public, commercial, or not-for-profit sectors.

\textbf{Statement } During the preparation of this work the author used chatGPT4o (through FAU's HAWKI) in order to improve the language in the abstract and introduction. After using this tool, the author reviewed and edited the content as needed and takes full responsibility for the content of the published article.

\textbf{Conflict of interest } The author declares no conflict of interest.

    \bibliographystyle{alpha}  
    \bibliography{bib.bib}

\appendix
\section{Natural Approaches To \Cref{prob:p4} That Do Not Work}\label{sec:failed-p4-approaches}
This appendix records several natural approaches, beyond the locking obstruction of \cref{prop:locking-no-go}, that do not prove \cref{prob:p4}. The results below are therefore \textit{not} lower bounds for
\(\sigma_{\mathrm{ap}}\). Thus they isolate the missing ingredient: a finite-information continuous measure-preserving spectral realization of a height-three source predicate.

Throughout the appendix, \(E_q\) denotes the set of \(q\)-th roots of unity, i.e.
\[
        E_q:=\{e^{2\pi i k/q} : k=0,\ldots,q-1\}\subset\mathbb T .
\]

\subsection{Dense Locking Fails For Continuous Maps}

A common adversarial idea is to run a tower on a reference map \(F_0\), collect the queried points, and modify the map away from those points. For continuous maps this fails as soon as the queried set is dense.

\begin{lemma}[Dense agreement forces equality]\label{lem:dense-agreement-forces-equality}
Let \(\X\) be a compact metric space, and let \(F,G: \X \to \X\) be continuous. If there is a dense set \(D\subseteq \X\) such that
\[
        F(x)=G(x) \qquad \forall x\in D,
\]
then
\[
        F=G \quad \text{on } \X.
\]
\end{lemma}

\begin{proof}
Let \(x\in \X\). Since \(D\) is dense, choose a sequence
\[
        x_n\in D, \, x_n\to x.
\]
By continuity,
\[
        F(x) = \lim_{n\to\infty}F(x_n) = \lim_{n\to\infty}G(x_n) = G(x).
\]
Since \(x\in \X\) was arbitrary, \(F=G\).
\end{proof}

\begin{corollary}[Countable locking cannot be used blindly]\label{cor:countable-locking-fails}
Let \(F_0: \X \to \X\) be continuous. Suppose a lower-bound construction attempts to build, for each code \(A\), a continuous map \(F_A: \X \to \X\) such that
\[
        F_A(x)=F_0(x)
\]
on every point queried by a two-limit tower along \(F_0\). If the union of these queried points is dense in \(\X\), then every such continuous \(F_A\) must satisfy
\[
        F_A=F_0.
\]
Consequently, this strategy cannot produce different approximate point spectra.
\end{corollary}

\begin{proof}
Apply \cref{lem:dense-agreement-forces-equality} with \(D\) equal to the countable union of queried points.
\end{proof}

\subsection{Baire-Category Arguments Do Not Apply To The Unrestricted Type-G}

Another natural idea is to prove the lower bound by showing that the target map has high Borel or Baire complexity. This does not work in the unrestricted type-\(G\) model, because finite general algorithms need not be Baire measurable.

\begin{proposition}[Finite type-\(G\) algorithms need not be Baire measurable]\label{prop:typeG-not-baire}
Let \(\X=\mathbb T=\mathbb R/\mathbb Z\) with Haar probability measure, and let
\[
        F_\theta(x)=x+\theta \pmod 1,
\]
for $\theta\in\mathbb T$. Let
\[
        \mathcal R:=\{F_\theta : \theta\in\mathbb T\}\subseteq\Omega_\X^m.
\]
Consider the restricted computational problem on \(\mathcal R\) with the point-value oracle \(F\mapsto F(x)\). Equivalently, since the evaluation set is required to be complex-valued, identify \(\mathbb T\) with the unit circle and include the character
\[
        \chi_1(x)=e^{2\pi i x}
\]
in the evaluation set. Then there exists a finite type-\(G\) algorithm
\[
        \Gamma:\mathcal R\to \mathcal M_H
\]
whose parameter map
\[
        \theta\mapsto\Gamma(F_\theta)
\]
is not Baire measurable.
\end{proposition}

\begin{proof}
Let \(B\subseteq\mathbb T\) be a set without the Baire property; see, for example, \cite[Ch.~4, Theorem~4.8]{oxtoby2013measure}. Define
\[
        K_0:=\{0\}, \, K_1:=\{1\}
\]
and set
\[
        \Gamma(F_\theta) :=
        \begin{cases}
        K_1,&\theta\in B,\\
        K_0,&\theta\notin B.
        \end{cases}
\]

We claim that \(\Gamma\) is a finite general algorithm on \(\mathcal R\). With the point-value oracle, it queries the single point \(0\). Since
\[
        F_\theta(0)=\theta,
\]
this one query determines \(\theta\) completely inside the restricted class \(\mathcal R\). In the complex-valued formulation, the single query
\[
        \chi_1(F_\theta(0))=e^{2\pi i\theta}
\]
also determines \(\theta\in\mathbb T\). Hence if another rotation \(F_\phi\) has the same queried value, then \(\phi=\theta\), and therefore
\[
        \Gamma(F_\phi)=\Gamma(F_\theta).
\]
Thus the consistency axiom for general algorithms is satisfied.

Let
\[
        \mathcal U:=\{K\in \mathcal M_H : d_H(K,K_1) < 1/3\}.
\]
Then
\[
        \{\theta\in\mathbb T:\Gamma(F_\theta)\in\mathcal U\}=B.
\]
Since \(B\) does not have the Baire property, the map
\[
        \theta\mapsto\Gamma(F_\theta)
\]
is not Baire measurable.
\end{proof}

\begin{remark}
Thus a Baire-category or Borel-rank proof may be appropriate for a regularity-restricted SCI model, such as a Borel, Baire, continuous, arithmetic, or Type-2 model. It does not prove a lower bound in unrestricted type \(G\).
\end{remark}

\subsection{Global Switches Encode The Hard Predicate Too Early}

A tempting construction is to choose two maps \(F_0,F_1\in\Omega_\X^m\) with Hausdorff-separated approximate point spectra and define
\[
        F_A=
        \begin{cases}
        F_0, & \widehat\Xi_3(A)=0,\\
        F_1, & \widehat\Xi_3(A)=1.
        \end{cases}
\]
This does not give a valid finite-information realization of \(\widehat\Xi_3\).

\begin{proposition}[Global switches are not finite-information realizations]\label{prop:global-switch-fails}
Let
\[
        \widehat\Xi_3:\{0,1\}^{\mathbb N}\to\{0,1\}
\]
be a source predicate which is not computable by a finite-coordinate general algorithm. Suppose \(\X\) is a set and \(A\mapsto F_A: \X \to \X\) is a coded family of maps. Assume that there exists
\[
        x_0\in \X, \, \varphi: \X \to\mathbb C,
\]
and two distinct complex numbers \(c_0\neq c_1\) such that
\[
        \varphi(F_A(x_0))= c_{\widehat\Xi_3(A)} \qquad \forall A\in\{0,1\}^{\mathbb N}.
\]
If the scalar value
\[
        A\mapsto\varphi(F_A(x_0))
\]
is determined by finitely many coordinates of \(A\) through a finite-coordinate general algorithm, then \(\widehat\Xi_3\) is finite-coordinate computable. Consequently, such a global switch cannot be a valid finite-information realization of a height-three source.
\end{proposition}

\begin{proof}
By the finite-information assumption, there is a finite-coordinate procedure which, on input \(A\), queries finitely many bits of \(A\) and outputs
\[
        \varphi(F_A(x_0))\in\{c_0,c_1\}.
\]
Since \(c_0\neq c_1\), we can post-process this output by the finite rule
\[
        c_0\mapsto0, \, c_1\mapsto1.
\]
The result is a finite-coordinate general algorithm for
\[
        A\mapsto\widehat\Xi_3(A),
\]
contradicting the assumption that \(\widehat\Xi_3\) is not finite-coordinate computable.
\end{proof}

\subsection{Finite Cyclic Blocks Only Detect Boundedness Of Periods}

\begin{lemma}[Closure of finite cyclic spectra]\label{lem:finite-cycles-collapse}
Let \((q_j)_{j\ge1}\) in $\mathbb{N}$, and define
\[
        K(q):=\overline{\bigcup_{j\ge1}E_{q_j}} \subseteq\mathbb T.
\]
Then
\[
        K(q) =
        \begin{cases}
        \displaystyle\bigcup_{q\in S}E_q,
        &\text{if }\{q_j:j\ge1\}=S\text{ is finite},\\[2mm]
        \mathbb T,
        &\text{if }(q_j)_{j\ge1}\text{ is unbounded}.
        \end{cases}
\]
\end{lemma}

\begin{proof}
If \((q_j)\) is bounded, then only finitely many integers occur among the \(q_j\). Hence
\[
        \bigcup_{j\ge1}E_{q_j} =
        \bigcup_{q\in S}E_q
\]
for a finite set \(S\subset\mathbb N\), and the right-hand side is finite and therefore closed.

Assume now that \((q_j)\) is unbounded. Let \(I\subset\mathbb T\) be a non-empty open arc. Choose \(j\) so large that
\[
        q_j>\frac{1}{|I|},
\]
where \(|I|\) denotes the normalized arc length. The grid \(E_{q_j}\) is \(1/q_j\)-dense in \(\mathbb T\), so \(E_{q_j}\cap I\neq\emptyset\). Thus every non-empty open arc intersects \(\bigcup_jE_{q_j}\). Hence this union is dense in \(\mathbb T\), and its closure is \(\mathbb T\).
\end{proof}

\begin{proposition}[Finite-cycle direct sums are too coarse for a height-three source]\label{prop:finite-cycles-too-coarse}
Consider a measure-preserving Koopman direct sum whose components are finite cycles with periods \(q_j(A)\), and suppose each \(q_j(A)\) is determined by finitely many bits of a code \(A\). Then the associated exact approximate point spectrum has the form
\[
        \sigma_{\mathrm{ap}}(\K_{F_A}) = \overline{\bigcup_{j\ge1}E_{q_j(A)}}.
\]
Consequently, the only nontrivial global dichotomy visible from the periods is whether the sequence \((q_j(A))\) is bounded or unbounded. The condition
\[
        \sup_j q_j(A)=\infty
\]
has the arithmetical form
\[
        \forall M\ \exists j\ [q_j(A)\ge M],
\]
where the bracketed predicate is finite-coordinate decidable. Thus this route naturally yields at most a \(\Pi^0_2\)-type source, not a \(\Pi^0_3\)-complete source.
\end{proposition}

\begin{proof}
The spectral formula follows because the Koopman operator of a \(q\)-cycle has eigenvalues \(E_q\), and the approximate point spectrum of an \(L^p\)-direct sum contains and is contained in the closure of the union of the block approximate point spectra. Hence
\[
        \sigma_{\mathrm{ap}}(\K_{F_A}) = \overline{\bigcup_{j\ge1}E_{q_j(A)}}.
\]
\Cref{lem:finite-cycles-collapse} shows that this compact set is finite if the periods are bounded and is the full unit circle if the periods are unbounded. If each \(q_j(A)\) is determined by finitely many bits of \(A\), then
\[
        q_j(A)\ge M
\]
is a finite-coordinate condition. Therefore
\[
        \sup_j q_j(A)=\infty \, \iff \, \forall M\ \exists j\ [q_j(A)\ge M]
\]
has only two-quantifier form over finite-coordinate predicates. Such a construction does not realize a \(\Pi^0_3\)-complete source unless the hard predicate has already been hidden in the definition of the periods, which would violate finite-information realizability.
\end{proof}

\subsection{Rank-One Odometers Are Too Coarse}

A second natural refinement is to replace finite cyclic blocks by odometers. We show now that rank-one odometers still do not carry enough spectral information for \cref{prob:p4}.

\begin{lemma}[Rank-one odometer spectra]\label{lem:rank-one-odometer-too-coarse}
Let \(b=(b_n)_{n\ge1}\) be a sequence of natural numbers and put
\[
        Q_n:=\prod_{\ell=1}^{n}b_\ell .
\]
For \(m\ge n\), let
\[
        \pi_{m,n}:\mathbb Z/Q_m\mathbb Z\to\mathbb Z/Q_n\mathbb Z
\]
be the canonical reduction map. Define the compact profinite abelian group
\[
        X_b := \varprojlim(\mathbb Z/Q_n\mathbb Z,\pi_{n+1,n})
\]
by
\[
        X_b = \left\{ (x_n)_{n\ge1} : x_n\in\mathbb Z/Q_n\mathbb Z,\ \pi_{n+1,n}(x_{n+1})=x_n\ \forall n \right\}.
\]
Let \(m_b\) be Haar probability measure on \(X_b\) and
\[
        \mathbf 1_b:=(1+Q_n\mathbb Z)_{n\ge1}\in X_b
\]
and define the odometer translation
\[
        T_bx:=x+\mathbf 1_b.
\]
Let
\[
        \mathcal U_b:= \K_{T_b}
\]
on \(L^p(X_b,m_b)\), for \(1<p<\infty\). Then
\[
        \sigma_{\mathrm{ap}}(\mathcal U_b) = \overline{\bigcup_{n\ge1}E_{Q_n}}.
\]
Consequently,
\[
        \sigma_{\mathrm{ap}}(\mathcal U_b) =
        \begin{cases}
        E_Q, &\text{if }Q_n=Q\text{ eventually},\\[1mm]
        \mathbb T, &\text{if }Q_n\to\infty.
        \end{cases}
\]
\end{lemma}

\begin{proof}
The sequence \((Q_n)\) is nondecreasing and \(Q_n\mid Q_{n+1}\). Hence either \(Q_n\) is eventually constant, or \(Q_n\to\infty\).

First suppose that \(Q_n=Q\) eventually. Then the inverse system stabilizes, and
\[
        X_b\cong\mathbb Z/Q\mathbb Z.
\]
Under this identification, \(T_b\) is the cyclic translation
\[
        x\mapsto x+1
\]
on the finite cyclic group \(\mathbb Z/Q\mathbb Z\). Therefore \(\mathcal U_b\) is the cyclic permutation operator of order \(Q\). Its eigenvalues are precisely $E_Q$. Since the underlying space is finite-dimensional in this case,
\[
        \sigma_{\mathrm{ap}}(\mathcal U_b)=\sigma(\mathcal U_b)=E_Q.
\]

Now suppose that \(Q_n\to\infty\). For \(n\in \mathbb{N}\) and \(k=0,\ldots,Q_n-1\), define
\[
        \chi_{n,k}(x) := \exp \left(\frac{2\pi i k x_n}{Q_n}\right), \, x=(x_m)_{m\ge1}\in X_b.
\]
This is a continuous character of \(X_b\), factoring through the quotient
\[
        X_b\to\mathbb Z/Q_n\mathbb Z.
\]
By the character-eigenvalue formula for compact-group rotations \cite[Proposition~14.24]{eisner2015operator}, applied to the compact group \(G=X_b\) and the rotation element \(a=\mathbf 1_b\), \(\chi_{n,k}\) is an eigenfunction of \(\mathcal U_b\) with eigenvalue
\[
        \chi_{n,k}(\mathbf 1_b) = e^{2\pi i k/Q_n}.
\]
Thus
\[
        \bigcup_{n\ge1}E_{Q_n} \subseteq \sigma_{\mathrm{ap}}(\mathcal U_b).
\]
Since \(Q_n\to\infty\), \cref{lem:finite-cycles-collapse} implies
\[
        \overline{\bigcup_{n\ge1}E_{Q_n}}=\mathbb T.
\]
Hence
\[
        \mathbb T\subseteq\sigma_{\mathrm{ap}}(\mathcal U_b).
\]

Conversely, \(T_b\) preserves the Haar measure, so \(\mathcal U_b\) is an \(L^p\)-isometry. For every unit vector \(f\in L^p(X_b,m_b)\),
\[
        \|(\mathcal U_b-zI)f\|_p \ge \bigl|\| \mathcal U_bf\|_p-|z|\|f\|_p\bigr| = \bigl|1-|z|\bigr|.
\]
Therefore
\[
        \sigma_{\inf}(\mathcal U_b-zI) \ge \operatorname{dist}(z,\mathbb T).
\]
In particular, if \(z\notin\mathbb T\), then
\[
        \sigma_{\inf}(\mathcal U_b-zI)>0,
\]
so \(z\notin\sigma_{\mathrm{ap}}(\mathcal U_b)\). Hence
\[
        \sigma_{\mathrm{ap}}(\mathcal U_b)=\mathbb T
\]
when \(Q_n\to\infty\).
\end{proof}

\begin{remark}
Thus rank-one odometers again produce only a finite-root-set vs. full-circle dichotomy. They do not provide a Hausdorff-separated coding of a \(\Pi^0_3\)-complete source. A more complicated Toeplitz or Bratteli-Vershik construction would have to encode essential-period information without reducing to the bounded/unbounded denominator dichotomy.
\end{remark}

\subsection{Known-Modulus Witnesses Cannot Prove The No-Modulus Lower Bound}

The wedge construction proves sharpness of a one-limit known-modulus fixed-\(\varepsilon\) upper bound. It cannot prove the no-modulus exact-spectrum lower bound.

\begin{proposition}[Known-modulus witnesses cannot prove \cref{prob:p4}]\label{prop:known-modulus-cannot-prove-p4}
Let \(\mathcal C\subseteq\Omega_\X^{\alpha,m}\) be any known-modulus measure-preserving subclass. The restriction of the exact approximate point spectrum problem to \(\mathcal C\) satisfies
\[
        \sigma_{\mathrm{ap}}(\K_F)\in\Pi_2^G \quad \text{on }\mathcal C.
\]
Therefore no lower bound proved solely on such a known-modulus subclass can imply
\[
        \sigma_{\mathrm{ap}}(\K_F)\notin\Pi_2^G \quad \text{on }\Omega_\X^m.
\]
\end{proposition}

\begin{proof}
By \cref{cor:ap-upper-bounds},
\[
        \sigma_{\mathrm{ap}}(\K_F)\in\Pi_2^G \quad \text{on }\Omega_\X^{\alpha,m}.
\]
Restricting a valid \(\Pi_2^G\)-tower from \(\Omega_\X^{\alpha,m}\) to any subclass
\[
        \mathcal C\subseteq\Omega_\X^{\alpha,m}
\]
gives a valid \(\Pi_2^G\)-tower on \(\mathcal C\). Hence
\[
        \sigma_{\mathrm{ap}}(\K_F)\in\Pi_2^G \quad \text{on }\mathcal C.
\]
Therefore a witness contained entirely in a known-modulus class cannot prove non-membership in \(\Pi_2^G\) for the no-modulus class \(\Omega_\X^m\).
\end{proof}

\subsection{Nonsingular Boundary Gadgets Do Not Solve The Measure-Preserving Exact Problem}

The boundary gadget used for
\[
        C_{\mathrm{ap},\varepsilon}(\K_F)
\]
in the nonsingular class exploits variation of the Radon-Nikodym density. This mechanism is unavailable in the measure-preserving class.

\begin{proposition}[Boundary-density gadgets cannot be directly transferred to \(\Omega_\X^m\)]\label{prop:boundary-gadget-not-mp}
Let \(F: \X \to \X\) be measure-preserving. Then
\[
        \rho_F:=\frac{dF_\#\omega}{d\omega}=1 \quad \text{a.e.}
\]
Consequently,
\[
        \|\K_F f\|_p=\|f\|_p
\]
for $f\in L^p(\X,\omega)$. Any construction whose spectral marker is produced by varying the values of
\[
        \rho_F^{1/p}
\]
cannot lie in \(\Omega_\X^m\).
\end{proposition}

\begin{proof}
This follows directly from \cref{prop:boundednessLp}.
\end{proof}


\end{document}